\documentclass{article}
\usepackage{amsmath, amssymb, amsthm}
\usepackage{mathrsfs}
\usepackage{graphicx}
\usepackage{tikz}
\usepackage[top=2.5cm, bottom=2.5cm, left=2.5cm, right=2.5cm]{geometry}

\newcommand{\be}{\begin{eqnarray}}
	\newcommand{\ben}{\begin{eqnarray*}}
		\newcommand{\en}{\end{eqnarray}}
	\newcommand{\enn}{\end{eqnarray*}}
\newcommand{\Z}{{\mathbb Z}}

\newcommand{\C}{{\mathbb C}}
\newcommand{\R}{{\mathbb R}}
\newcommand{\G}{{\Gamma}}
\newtheorem{definition}{Definition}
\newtheorem{theorem}{Theorem}
\newtheorem{lemma}{Lemma}
\newtheorem{assumption}{Assumption}
\newtheorem{remark}{Remark}

\begin{document}
\title{Limiting absorption principle for time-harmonic elastic scattering of plane waves from diffraction gratings}
\author{Jianli Xiang\thanks{Three Gorges Mathematical Research Center, College of Mathematics and Physics, China Three Gorges University, Yichang 443002, China. Email: \texttt{xiangjianli@ctgu.edu.cn}} \quad
	and Guanghui Hu\thanks{Corresponding author, School of Mathematical Sciences and LPMC, Nankai University, Tianjin 300071, China. Email: \texttt{ghhu@nankai.edu.cn}}}
\date{}
\maketitle
\begin{abstract}
  We establish the limiting absorption principle for time-harmonic elastic scattering of plane waves by a periodic rigid diffraction grating. By perturbing the frequency with a small positive imaginary part, we regularize the ill-posed problem at propagative wavenumbers (that is, when uniqueness fails under the classical Rayleigh expansion condition) and characterize the limiting solution via a singular perturbation result from functional analysis. The limiting solution satisfies the original scattering problem together with an additional constraint that ensures uniqueness. Both incident pressure and shear waves are considered, and the same constraint condition is obtained in both cases. The results provide a rigorous selection mechanism for physically admissible solutions at resonance frequencies. Our framework extends naturally to the Neumann (cavity) boundary condition as well as other transmission conditions, in particular when guided waves exist in periodic structures.
\end{abstract}



\section{Introduction}

The scattering of acoustic, electromagnetic and elastic waves by periodic structures has been a subject of intensive research over the past several decades, owing to its wide range of applications in physics and engineering, including diffraction gratings, photonic and phononic crystals, surface acoustic wave devices, seismic exploration, and non-destructive testing. Periodic structures such as phononic crystals and metamaterials exhibit unique wave propagation properties, including band gaps, negative refraction, and waveguiding, which have attracted considerable attention in both theoretical and experimental studies. The propagation of elastic waves in periodic media is also of fundamental importance in geophysics, where layered and periodically structured geological formations significantly affect seismic wave propagation~\cite{AkiRichards}. In addition, surface acoustic wave (SAW) devices, which rely on the interaction of elastic waves with periodic gratings, are widely used in modern telecommunications and sensor technologies~\cite{Morgan}. The mathematical modeling and analysis of such scattering problems not only provide a deeper understanding of the underlying physics but also offer rigorous theoretical foundations for the design and optimization of practical devices.

The rigorous mathematical analysis of elastic wave scattering by a periodic rigid surface began with the work of Arens~\cite{Arens1999,ArensT1999} for two-dimensional diffraction gratings and~\cite{Arens2001,Arens2002} for more general rough surfaces. In particular,~\cite{Arens1999} established existence and uniqueness of quasi-periodic solutions to the Dirichlet problem for grating profiles given by smooth $(C^2)$ periodic functions, using the boundary integral equation method. Subsequently, Elschner and Hu in~\cite{EH2010,EH2012} investigated the same Dirichlet problem in general Lipschitz domains via the variational approach. It was shown in~\cite{EH2010,EH2012} that, for either an incident pressure or shear wave, a quasi-periodic solution to the equivalent variational formulation always exists. Moreover, uniqueness is guaranteed if the grating profile is given by a Lipschitz graph in $\mathbb{R}^2$ or $\mathbb{R}^3$.
However, all the aforementioned well-posedness results are established under the assumption that the quasi-periodic momentum is not a propagative wavenumber, that is, under the so-called non-resonance condition (see Def. \ref{def1} (ii)). At resonance frequencies, where the quasi-momentum coincides with a propagative wavenumber, the Fredholm alternative fails to guarantee uniqueness, and the Rayleigh expansion condition alone is insufficient to single out a unique physically admissible solution. Such resonance phenomena are ubiquitous in elastic wave propagation in periodic media and play a crucial role in the design of waveguides, filters, and energy harvesting devices. Therefore, a rigorous understanding of the resonance scattering mechanism is of both theoretical and practical importance.

The limiting absorption principle (LAP) provides a powerful and systematic framework for selecting the physically correct solution at resonance frequencies. The fundamental idea of the LAP is to introduce a small artificial absorption into the medium, which renders the otherwise ill-posed stationary problem well-posed, and then to recover the physical solution by passing to the limit as the absorption parameter tends to zero. This principle is deeply rooted in the physical intuition that any realistic medium exhibits some degree of dissipation, and the physically admissible solution should be obtained as the limiting case of vanishing dissipation. The LAP was first developed in the context of quantum scattering theory and has since become a standard tool in the analysis of wave propagation in unbounded domains; see, e.g., the classical works of Eidus~\cite{Eidus} and Vainberg~\cite{Vainberg}. For open waveguide scattering, the LAP and its connection to radiation conditions have been extensively studied by Nosich~\cite{Nosich1994}, who clarified the relationship between the limiting absorption principle, the limiting amplitude principle, and the physical radiation conditions in open waveguides. The asymptotic behavior and radiation conditions for time-harmonic waves in periodic waveguides were further investigated by Fliss and Joly~\cite{FJ16}, who provided a rigorous characterization of the outgoing wave fields in such structures. More recently, Hoang~\cite{Hoang2011} established the LAP for acoustic scattering in semi-infinite periodic waveguides.

The works \cite{Kirsch2023,KL18} by Kirsch and Lechleiter developed a general framework for the LAP in scattering by an open periodic layer, deriving a radiation condition that characterizes the outgoing nature of the scattered field. An abstract singular perturbation lemma was established in ~\cite{Kirsch2023} to characterize the limiting solution via an additional constraint condition when the underlying operator is not invertible at the limit, which serves as a key ingredient in the present work. In the context of acoustic and electromagnetic scattering of plane waves by periodic structures, the LAP has been further investigated in~\cite{HK23,HuKiZh}, where a comprehensive analysis of the limiting absorption principle for plane wave scattering of periodic structures is presented.
For elastic wave scattering, however, the application of the limiting absorption principle presents additional challenges compared to the scalar acoustic case or the electromagnetic case. The Navier system is inherently more complex due to the coupling of compressional and shear waves, each with different wave speeds, which leads to a vectorial nature of the problem and a more intricate DtN map. Moreover, the variational formulation involves a sesquilinear form that is not as straightforward as in the scalar case, and the derivation of the limiting constraint condition requires a careful analysis of the self-adjointness and positivity of the associated frequency-independent operators. Unlike the acoustic or electromagnetic cases, where the limiting constraint condition often takes a relatively simple form, the elastic case yields an integral identity that couples both components of the displacement field through the Lam\'e parameters and the incident angle.

In this work, we establish the limiting absorption principle for time-harmonic elastic scattering of plane waves by a rigid diffraction grating in $\mathbb{R}^2$. The grating interface is assumed to be a $2\pi$-periodic non-selfintersecting Lipschitz curve, and the medium above it is homogeneous and isotropic.
Even in such a simple background setting, surface waves cannot usually be ruled out.
We treat both incident pressure waves and incident shear waves separately, and show that in both cases the limiting absorption process yields a unique solution, provided that an additional constraint condition is imposed when the quasi-periodic momentum is a propagative wavenumber. Remarkably, the two cases give rise to the same integral constraint condition (see Theorems~\ref{THP} and \ref{THS}), which takes the form of an orthogonality relation over the infinite periodic cell and involves the Lam\'e constants and the kernal space to the homogeneous scattering problem. This constraint condition serves as a natural selection mechanism for the physically admissible solution at resonance.

The main results of this paper are summarized as follows. In Section~\ref{sec:classical}, we recall the classical formulation of the scattering problem and the variational approach, together with the existing well-posedness results under the non-resonance assumption. In Section~\ref{sec:LAP}, we present the limiting absorption principle for the elastic grating problem. We first introduce the abstract functional-analytic lemma (Lemma~\ref{LAP}) and then verify its assumptions for the elastic operator. The pressure and shear wave cases are treated separately in Subsections~\ref{subsec:pressure} and~\ref{subsec:shear}, respectively, where we derive the additional constraint condition in each case and establish the main existence and uniqueness results (Theorem~\ref{THP} and Theorem~\ref{THS}).

\section{Scattering of elastic waves from a rigid periodic surface: classical results}\label{sec:classical}
\subsection{Problem formulation}
Let the profile of the diffraction grating be given by a non-selfintersecting Lipschitz curve $\Lambda \subset \mathbb{R}^2$, which is supposed to be $2\pi$-periodic in $x_{1}$. Let $D$ be the unbounded domain above $\Lambda$. We assume the region $D$ is filled with an isotropic, homogenous elastic medium characterized by the Lam\'e constants $\lambda$, $\mu$ satisfying $\mu>0$, $\lambda+\mu>0$. For simplicity the mass density of the elastic medium will be normalized to be one.
We assume that a time-harmonic elastic plane wave $u^{\text {in}}$ with incident angle $\theta\in(-\pi/2,\pi/ 2)$ is incident onto $\Lambda$ from above. The incoming wave can be either an incident pressure wave taking the form
\begin{equation*}
u^{\text{in}}=u_p^{\text {in}}(x)=\hat{\theta} \exp \left(i k_p \hat{\theta}\cdot x\right) \quad \text { with } \hat{\theta}:=(\sin\theta,-\cos\theta),
\end{equation*}
or an incident shear wave of the form
\begin{equation*}
u^{\text{in}}=u_s^{\text{in}}(x)=\hat{\theta}^{\perp} \exp \left(i k_s \hat{\theta} \cdot x\right) \quad \text { with } \hat{\theta}^{\perp}:=(\cos\theta,\sin\theta) .
\end{equation*}
Here,
\begin{equation*}
k_p:=\omega/ \sqrt{2\mu+\lambda}, \quad k_s:=\omega/\sqrt{\mu}
\end{equation*}
are the compressional and shear wave numbers, respectively, and $\omega>0$ stands for the angular frequency of the harmonic motion.

The propagation of the time-harmonic total elastic wave in $D$ is governed by the Navier equation (or system)
\begin{equation} \label{a0}
\left(\Delta^*+\omega^2\right)u=0 \text { in } D, \quad \Delta^*:=\mu\Delta+(\lambda+\mu) \text { grad\,div},
\end{equation}
where $u=u^{\text{in}}+u^{\text{sc}}$ is the total displacement field and $u^{\text{sc}}$ denotes the scattered field.  Moreover, we require that the total field satisfies the rigid (Dirichlet) boundary condition
\begin{equation} \label{a1}
u=0 \quad {\rm on} \quad \Lambda .
\end{equation}
The periodicity of the structure, together with the form of the incident waves, implies that the solution $u$ must be quasiperiodic with phase-shift $\alpha$ (or $\alpha$-quasiperiodic), i.e.
\begin{equation} \label{a2}
u\left(x_1+2 \pi, x_2\right)=\exp \left(2i\alpha\pi x_1\right) u\left(x_1,x_2\right), \quad\left(x_1,x_2\right) \in D,
\end{equation}
where quasi-periodic momentum $\alpha$ depends on the form of the incident plane wave as follows
\be\label{alpha}
\alpha:=\left\{\begin{array}{lll}
k_p\sin\theta\quad&&\mbox{if}\quad u^{\text{in}}=u^{\text{in}}_p,\\ [2pt]
k_s \sin \theta \quad&&\mbox{if}\quad u^{\text{in}}=u^{\text{in}}_p.
\end{array}\right.
\en
To ensure well-posedness of the boundary value problem \eqref{a0}-\eqref{a2}, a radiation condition must be imposed as $x_2 \rightarrow+\infty$. First we note that the scattered field $u^{\text{sc}}$, which also satisfies the Navier equation \eqref{a0}, can be decomposed in $D$ as
\begin{equation} \label{a3}
u^{\text{sc}}=\frac{1}{i}({\rm grad}\, \varphi+\overrightarrow{{\rm curl}}\, \psi) \quad \text { with } \varphi:=-\frac{i}{k_p^2} \,{\rm div}\, u^{\text{sc}}, \quad \psi:=\frac{i}{k_s^2} \, {\rm curl}\, u^{\text{sc}},
\end{equation}
where the two curl operators in $\mathbb{R}^2$ are defined by
\begin{equation*}
{\rm curl}\, u:=\partial_1 u_2-\partial_2 u_1, \quad u=\left(u_1,u_2\right)^{\top} \quad {\rm and }\quad \overrightarrow{{\rm curl}}\, v:=\left(\partial_2 v,-\partial_1 v\right)^{\top} ,
\end{equation*}
and the scalar functions $\varphi, \psi$ satisfy the homogeneous Helmholtz equations
\begin{equation*}
\left(\Delta+k_p^2\right) \varphi=0 \quad {\rm and} \quad \left(\Delta+k_s^2\right) \psi=0 \quad {\rm in }~D .
\end{equation*}
Here and in the following the notation $\partial_j v=\partial v/\partial x_j$ is used. Note that the above relations follow from the well-known decomposition of the scattered field $u^{{\rm sc}}$ into its compressional and shear parts,
\begin{equation*}
u^{{\rm sc}}=u_p+u_s, \quad u_p:=-\frac{1}{k_p^2}\, {\rm grad \, div}\, u^{{\rm sc}}, \quad u_s:=\frac{1}{k_s^2}\, \overrightarrow{{\rm curl}}\, {\rm curl}\, u^{{\rm sc}},
\end{equation*}
and the fact that $u^{{\rm sc}}$ satisfies equation \eqref{a0}.

Now, as $\varphi$ and $\psi$ are $\alpha$-quasiperiodic solutions to the Helmholtz equation in the unbounded domain $D$, we impose the classical outgoing wave condition of the Helmholtz equation on them. For $x_2>\Lambda^{+}:=\max\limits_{\left(x_1, x_2\right) \in \Lambda} x_2$, we assume that $\varphi$, $\psi$ have upward $\alpha$-quasiperiodic Rayleigh expansions of the form
\begin{equation*}
\varphi(x)=\sum_{n \in \mathbb{Z}} A_{p, n} \exp \left(i \alpha_n x_1+i \beta_n x_2\right), \quad \psi(x)=\sum_{n \in \mathbb{Z}} A_{s, n} \exp \left(i \alpha_n x_1+i \gamma_n x_2\right),
\end{equation*}
where the constants $A_{p,n}$, $A_{s,n}\in\mathbb{C}$ are called Rayleigh coefficients and
\begin{equation*}
\alpha_n:=\alpha+n, \quad \beta_n:= \begin{cases}
\sqrt{k_p^2-\alpha_n^2} & \text { if }\left|\alpha_n\right| \leq k_p, \\ i \sqrt{\alpha_n^2-k_p^2} & \text { if }\left|\alpha_n\right|>k_p,
\end{cases}
\end{equation*}
and $\gamma_n$ is defined analogously as $\beta_n$ with $k_p$ replaced by $k_s$. It follows from \eqref{a3} that the two components of the scattered field $u^{{\rm sc}}=(u_1^{{\rm sc}}, u_2^{{\rm sc}})$ in $D$ can be represented as
\begin{equation*}
u_1^{{\rm sc}}=\frac{1}{i}\left(\partial_1 \varphi+\partial_2 \psi\right), \quad u_2^{{\rm sc}}=\frac{1}{i}\left(\partial_2 \varphi-\partial_1 \psi\right).
\end{equation*}
Therefore, we finally obtain a corresponding expansion of $u^{{\rm sc}}$ into outgoing plane elastic waves:
\begin{equation} \label{a4}
u^{{\rm sc}}(x)=\sum_{n \in \mathbb{Z}}\left\{A_{p,n}\left(\begin{array}{c}
	\alpha_n \\   \beta_n \end{array}\right) e^{i\alpha_n x_1+i\beta_n x_2}+A_{s, n} \left(\begin{array}{c}
	\gamma_n \\  -\alpha_n \end{array}\right) e^{i\alpha_n x_1+i\gamma_n x_2}\right\},
\end{equation}
for $x_2>\Lambda^{+}$. This is the classical radiation condition used in the literature. Since $\beta_n$ and $\gamma_n$ are real for at most a finite number of indices, only a finite number of plane waves in \eqref{a4} propagate into the far field, with the remaining evanescent waves (or surface waves) decaying exponentially as $x_2 \rightarrow+\infty$. The above expansion converges uniformly with all derivatives in the half-plane $\left\{x\in \mathbb{R}^2: x_2 \geq b\right\}$, for any $b>\Lambda^{+}$, and the Rayleigh coefficients are uniquely determined by the Fourier coefficients $u_n$ of the function $\exp\left(-i\alpha x_1\right)u^{{\rm sc}}\left(x_1,b\right)$:
\begin{equation}\label{un}
u_n=D_n\left(\begin{array}{c}
		A_{p,n} \exp \left(i \beta_n b\right) \\
		A_{s,n} \exp \left(i \gamma_n b\right)
\end{array}\right), \quad D_n:=\left(\begin{array}{cc}
		\alpha_n & \gamma_n \\
		\beta_n & -\alpha_n
\end{array}\right) .
\end{equation}
Note here that ${\rm det}\,D_n=-\left(\alpha_n^2+\beta_n \gamma_n\right)\neq 0$ for all $n \in \mathbb{Z}$. Our diffraction problem can now be formulated as the following boundary value problem.
\begin{description}
	\item[
Dirichlet problem (DP):] Given a grating profile curve $\Lambda \subset \mathbb{R}^2$ (which is $2\pi$-periodic in $x_1$) and an incident field $u^{{\rm in}}$ of the form $u_{p}^{{\rm in}}$ or $u_{s}^{{\rm in}}$, find a vector function $u=u^{{\rm in}}+u^{{\rm sc}} \in [H_{{\rm loc}}^1(D)]^2$ that satisfies \eqref{a0}-\eqref{a2} and the radiation condition \eqref{a4}.
\end{description}

While the Dirichlet boundary condition and a single incident wave are adopted here for clarity, our analysis applies equally to other boundary conditions and to the superposition of multiple incident elastic plane waves (see Remarks \ref{rem2} and \ref{rem3}).

\subsection{Variational formultion in the non-resonant regime}
Following the variational approach in \cite{EH2010}, we recall the equivalent variational formulation of the boundary value problem (DP), which is posed in a truncated bounded periodic cell in $\mathbb{R}^2$ and is enforcing the radiation condition on the artificial boundary. Introduce an artificial boundary
\begin{equation*}
\Gamma_b:=\left\{\left(x_1, b\right): 0 \leq x_1 \leq 2 \pi\right\}, \quad b>\Lambda^{+},
\end{equation*}
and the truncated bounded domain
\begin{equation*}
\Omega_b=\Omega_{\Lambda, b}:=\left\{\left(x_1, x_2\right) \in D: 0<x_1<2 \pi, x_2<b\right\},
\end{equation*}
lying between the segment $\Gamma_b$ and one period of the grating profile curve which we denote by $\Lambda$ again (see Figure \ref{fig}).
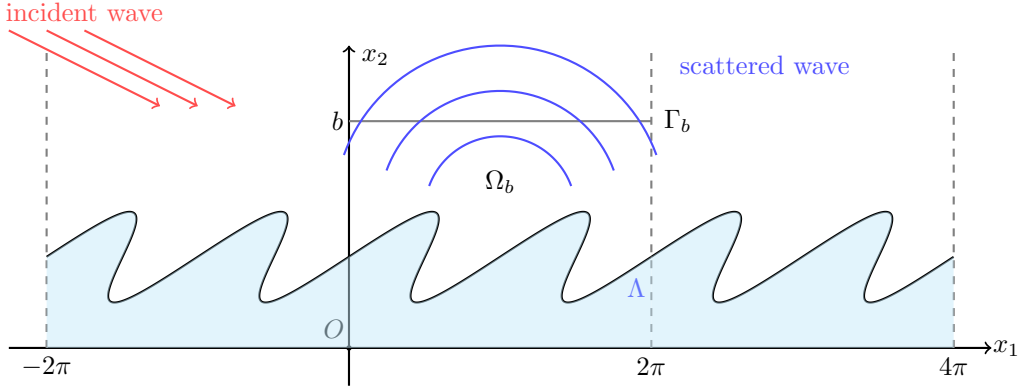
\begin{figure}[h]
  \centering
  \begin{tikzpicture}
  \filldraw[black] (0,0) circle (0.8pt);
  \draw (-0.2,0.25) node{$O$};
  \draw[thick,->] (-4.5,0)--(8.5,0);
  \draw (8.7,0) node{$x_{1}$};
  \draw[thick,->] (0,-0.5)--(0,4);
  \draw (0.35,3.85) node{$x_{2}$};
  \draw[gray,thick] (0,3)--(4,3);
  \draw (4.35,3) node{$\Gamma_b$};
  \draw (-0.15,3) node{$b$};
  \draw[gray,thick,dashed] (4,0)--(4,4);
  \draw (4,-0.25) node{$2\pi$};
  \draw[gray,thick,dashed] (-4,0)--(-4,4);
  \draw (-4,-0.25) node{$-2\pi$};
  \draw[gray,thick,dashed] (8,0)--(8,4);
  \draw (8,-0.25) node{$4\pi$};
  \draw[black,thick, smooth, domain=0:2*pi, samples=300, variable=\t]
	plot ({2*\t/pi + 0.6*sin(2*\t r) - 4},
	{1.2 + 0.6*sin(2*\t r + 0.5)});
  \draw[black,thick, smooth, domain=0:2*pi, samples=300, variable=\t]
	plot ({2*\t/pi + 0.6*sin(2*\t r)},
	{1.2 + 0.6*sin(2*\t r + 0.5)});
  \draw (3.8,0.8) node[blue]{$\Lambda$};
  \draw[black,thick, smooth, domain=0:2*pi, samples=300, variable=\t]
	plot ({2*\t/pi + 0.6*sin(2*\t r) + 4},
	{1.2 + 0.6*sin(2*\t r + 0.5)});
  \fill[cyan!20, opacity=0.5, smooth, domain=0:2*pi, samples=300, variable=\t]
	plot ({2*\t/pi + 0.6*sin(2*\t r) - 4},
	{1.2 + 0.6*sin(2*\t r + 0.5)})
	-- (0,0) -- (-4,0) -- cycle;
  \fill[cyan!20, opacity=0.5, smooth, domain=0:2*pi, samples=300, variable=\t]
	plot ({2*\t/pi + 0.6*sin(2*\t r)},
	{1.2 + 0.6*sin(2*\t r + 0.5)})
	-- (4,0) -- (0,0) -- cycle;
  \fill[cyan!20, opacity=0.5, smooth, domain=0:2*pi, samples=300, variable=\t]
	plot ({2*\t/pi + 0.6*sin(2*\t r) + 4},
	{1.2 + 0.6*sin(2*\t r + 0.5)})
	-- (8,0) -- (4,0) -- cycle;
  \draw (2,2.2) node{$\Omega_b$};
  \draw[red!70] (-3.5,4.2) node[above]{incident wave};
  \draw[red!70, thick, ->] (-4.0,4.2) -- (-2.0,3.2);
  \draw[red!70, thick, ->] (-4.0+0.5,4.2) -- (-2.0+0.5,3.2);
  \draw[red!70, thick, ->] (-4.0-0.5,4.2) -- (-2.0-0.5,3.2);
  \draw[blue!70, thick, domain=20:160, samples=50, variable=\ang]
  plot ({1.0*cos(\ang) + 2}, {1.0*sin(\ang) + 1.8});
  \draw[blue!70, thick, domain=20:160, samples=50, variable=\ang]
  plot ({1.6*cos(\ang) + 2}, {1.6*sin(\ang) + 1.8});
  \draw[blue!70, thick, domain=20:160, samples=50, variable=\ang]
  plot ({2.2*cos(\ang) + 2}, {2.2*sin(\ang) + 1.8});
  \draw[blue!70] (5.5,3.5) node[above]{scattered wave};	
  \end{tikzpicture}
  \caption{Illustration of the diffraction problem for a plane wave incidence and $\Gamma_b$ , $\Omega_b$, $\Lambda$. The interface is a non-selfintersecting Lipschitz rigid curve which may admit surface waves to propagate along the $x_1$-direction.  }\label{fig}
\end{figure}
Let $H_\alpha^1\left(\Omega_b\right)$ denote the Sobolev space of scalar functions over $\Omega_b$ which are $\alpha$-quasiperiodic with respect to $x_1$. We introduce the energy space
\begin{equation*}
V_\alpha\left(\Omega_b\right):=\left\{u\in [H_\alpha^1\left(\Omega_b\right)]^2: \left.u\right|_{\Lambda}=0\right\}.
\end{equation*}
In the following $V_\alpha\left(\Omega_b \right)$ will be equipped with the norm in the usual Sobolev space $[H_\alpha^1\left(\Omega_b \right)]^2$ of vector functions.
We still need the following periodic spaces and the trace spaces on $\G_b$:
\ben
&&	V_{{\rm per}}\left(\Omega_b\right):=\left\{u\in [H^1\left(\Omega_b\right)]^2: u(x_{1},b) ~{\rm is }~ 2\pi\text{-periodic in }x_{1} \text{ and } \left.u\right|_{\Lambda}=0\right\},\\
&&	[L_{{\rm per}}^{2}(\Omega_{b})]^{2}:=\left\{u\in [L^2\left(\Omega_b\right)]^2: u(x_{1},b) ~{\rm is }~ 2\pi\text{-periodic in }x_{1} \right\},\\
&&[H_{\alpha}^{1/2}(\Gamma_{b})]^2:=\left\{f\in [H^{1/2}(\Gamma_b)]^2: e^{-i\alpha x_{1}}f(x_{1}) ~{\rm is }~ 2\pi\text{-periodic in }x_{1} \right\}, \\
&&[H_{{\rm per}}^{1/2}(\Gamma_{b})]^2:=\left\{f\in [H^{1/2}(\Gamma_b)]^2: f(x_{1}) ~{\rm is }~ 2\pi\text{-periodic in }x_{1} \right\}.
\enn
Next we introduce the sesquilinear form $B(u,\varphi)$ defined over quasi-periodic spaces by
\begin{equation*}
\widetilde{B}_{\omega}(u,\varphi):=\int_{\Omega_b}\left(E(u, \overline{\varphi})-\omega^2 u \cdot \overline{\varphi}\right) \mathrm{d} x-\int_{\Gamma_b} \overline{\varphi} \cdot \mathscr{\widetilde{T}}_{\omega} u \mathrm{~d} s \quad \text{for all }u, \phi \in V_{\alpha}\left(\Omega_b\right),
\end{equation*}
with the quasi-periodic Dirichlet-to-Neumann (DtN) map $\mathscr{\widetilde{T}}_{\omega}u :=\mathscr{\widetilde{T}}_{\omega} \left(\left.u\right|_{\Gamma_b}\right)$ on the artificial boundary, and
\begin{align*}
E(v,u):=&(2\mu+\lambda)\Big(\frac{\partial v_{1}}{\partial x_{1}}\frac{\partial u_{1}}{\partial x_{1}}+\frac{\partial v_{2}}{\partial x_{2}}\frac{\partial u_{2}}{\partial x_{2}}\Big)+\mu\Big(\frac{\partial v_{1}}{\partial x_{2}}\frac{\partial u_{1}}{\partial x_{2}}+\frac{\partial v_{2}}{\partial x_{1}}\frac{\partial u_{2}}{\partial x_{1}}\Big) \\
&+\lambda\Big(\frac{\partial v_{1}}{\partial x_{1}}\frac{\partial u_{2}} {\partial x_{2}}+\frac{\partial v_{2}}{\partial x_{2}}\frac{\partial u_{1}}{\partial x_{1}}\Big)+\mu\Big(\frac{\partial v_{1}}{\partial x_{2}}\frac{\partial u_{2}}{\partial x_{1}}+\frac{\partial v_{2}}{\partial x_{1}}\frac{\partial u_{1}}{\partial x_{2}}\Big).
\end{align*}
Note that, the frequency-dependent map $\mathscr{\widetilde{T}}_{\omega}$ takes the form (see \cite{EH2010})
\begin{equation*}
(\mathscr{\widetilde{T}}_{\omega} w)(x_{1})=-\sum_{n\in\mathbb{Z}} W_n(\omega) \widetilde{w}_n \exp \left(i\alpha_n x_1\right),\quad w(x_{1})=\sum_{n\in\mathbb{Z}}\widetilde{w}_n \exp\left(i\alpha_n x_1\right) \in[H^{1/2}_{\alpha}(\Gamma_{b})]^{2},
\end{equation*}
where
\begin{equation*}
W_n(\omega):=\frac{1}{i}\left(\begin{array}{cc}
	\omega^{2}\beta_n/d_{n} & 2\mu\alpha_{n}-\omega^{2}\alpha_{n}/d_{n} \\
	-2\mu\alpha_{n}+\omega^{2}\alpha_{n}/d_{n} & \omega^{2}\gamma_n/d_{n}
\end{array}\right), \quad  d_{n}:=\alpha_{n}^{2}+\beta_n\gamma_{n}.
\end{equation*}

Introduce the stress vector or traction $Tu$:
\begin{equation*}
Tu:=2\mu \partial_{\nu}u+\lambda~\nu~\nabla\cdot u+\mu\,\nu\;\overrightarrow{{\rm curl}}~u 
=\left(\begin{array}{ccc}(\lambda+2\mu)\frac{\partial u_{1}}{\partial x_{1}}+\lambda \frac{\partial u_{2}}{\partial x_{2}}
&\mu\Big(\frac{\partial u_{1}}{\partial x_{2}}+\frac{\partial u_{2}}{\partial x_{1}}\Big)\\
\mu\Big(\frac{\partial u_{1}}{\partial x_{2}}+\frac{\partial u_{2}}{\partial x_{1}}\Big)
&\lambda\frac{\partial u_{1}}{\partial x_{1}}+(\lambda+2\mu)\frac{\partial u_{2}}{\partial x_{2}}  \end{array}\right)\nu,
\end{equation*}
where $\nu$ denotes the exterior unit normal on the boundary of $\Omega_{b}$. Moreover, we have
\begin{equation*}
Tu=2\mu\partial_{2} u+\lambda\left(\begin{array}{c}0 \\1 \end{array}\right)(\partial_{1} u_{1}+\partial_{2}u_{2})+\mu\left(\begin{array}{c}1 \\ 0 \end{array}\right)(\partial_{1} u_{2}-\partial_{2}u_{1}) \quad {\rm on} \quad \Gamma_{b}.
\end{equation*}

Applying Betti's identity to a solution $u=u^{{\rm sc}}+u^{{\rm in}}$ of (DP), one obtains
\begin{equation*}
(Tu)|_{\Gamma_b}=[T\left(u^{{\rm sc}}+u^{{\rm in}}\right)]|_{\G_b}=
\mathscr{\widetilde{T}}_{\omega}(u^{{\rm sc}}|_{\G_b})+(T u^{{\rm in }})|_{\G_b} =\mathscr{\widetilde{T}}_{\omega} (u|_{\G_b})+g_0,
\end{equation*}
with
\ben
g_0:=(T u^{{\rm in }})|_{\G_b} -\mathscr{\widetilde{T}}_{\omega} (u^{{\rm in}}|_{\G_b}).
\enn
Hence, the variational formulation of (DP) reads as follows: Find $u\in V_\alpha\left(\Omega_b \right)$ such that
\begin{equation} \label{var0}
\widetilde{B}_{\omega}(u,\varphi)=\int_{\Gamma_b} f_0 \cdot\overline{\varphi}\, \mathrm{d}s \quad \text{for all }\varphi \in V_\alpha\left(\Omega_b\right) .
\end{equation}
Here
\begin{equation*}
g_0=g_{p,0}:=\frac{2i\beta_0 k_p(\lambda+2 \mu)}{d_0}\left(\begin{array}{l}
		-\alpha \\  \gamma_0
\end{array}\right) \exp\left(i\alpha x_1-i\beta_0 b\right),\quad  \alpha:=k_p\sin\theta
\end{equation*}
for an incident pressure wave of the form $u_{p}^{{\rm in}}=\hat{\theta} \exp \left(i k_p \hat{\theta}\cdot x\right)$, and
\begin{equation}\label{shear}
g_0=g_{s,0}:=-\frac{2i\gamma_0 k_s \mu}{d_0}\left(\begin{array}{c}
		\beta_0 \\  \alpha
\end{array}\right) \exp\left(i\alpha x_1-i\gamma_0 b\right),\quad \alpha:=k_s \sin \theta
\end{equation}
for an incident shear wave of the form $u_{s}^{{\rm in}}=\hat{\theta}^{\perp} \exp \left(i k_s \hat{\theta} \cdot x\right)$.

Define the periodic DtN map $\mathscr{T}_{\omega}$ on the artificial boundary $\Gamma_{b}$ by
\begin{equation*}
(\mathscr{T}_{\omega} v)(x_{1},b)=-\sum_{n\in\mathbb{Z}} W_n(\omega) v_n e^{i n x_1},\quad v(x_{1})=\sum_{n\in\mathbb{Z}}v_n e^{i n x_1} \in[H^{1/2}_{{\rm per}}(\Gamma_{b})]^{2}.
\end{equation*}
Defining $v:=e^{-i\alpha x_{1}}u$ and $\psi:=e^{-i\alpha x_{1}}\varphi$, we get the periodic form of the variational formulation for the incident pressure wave as follows: find $v\in V_{{\rm per}}\left(\Omega_b\right)$ such that
\be\label{Va-p}
B_{\omega}(v,\psi)=F_\omega(\psi) \quad \text{for all }\psi\in V_{{\rm per}}\left(\Omega_b\right),
\en
where now
\begin{align*}
B_{\omega}(v,\psi):=&\int_{\Omega_b}\left(E(v, \overline{\psi})-(\omega^2-\mu\alpha^{2})v \cdot \overline{\psi}-2i\alpha\mu\frac{\partial v}{\partial x_{1}}\cdot\overline{\psi}\right) \mathrm{d} x-\int_{\Gamma_b} \overline{\psi} \cdot \mathscr{T}_{\omega}v \mathrm{~d} s \\
&-i\alpha(\lambda+\mu)\int_{\Omega_b}\left[\left(\frac{\partial v_{2}}{\partial x_{2}}+2\frac{\partial v_{1}}{\partial x_{1}}\right)\overline{\psi}_{1}+ \frac{\partial v_{1}}{\partial x_{2}}\overline{\psi}_{2}\right] \mathrm{d}x +(\lambda+\mu)\int_{\Omega_b}\alpha^{2}v_{1}\overline{\psi}_{1}\mathrm{d}x.
\end{align*}
For incident pressure waves,  we have $\alpha=k_p\sin\theta$ and
\be\label{R-p}
F_\omega(\psi)=\frac{2i\beta_0 k_p(\lambda+2 \mu)}{d_0}e^{-i\beta_0 b} \int_{0}^{2\pi} \left(\begin{array}{l}
 -\alpha \\  \gamma_0 \end{array}\right) \cdot\overline{\psi(x_{1},b)}\,\mathrm{d}x_{1}\quad \text{for all }\psi\in V_{{\rm per}}\left(\Omega_b\right).
\en
The periodic variational formulation for the incident shear waves can be formulated analogously by changing the right hand side of \eqref{R-p} and using $\alpha=k_s\sin\theta$.

To proceed, we introduce the concepts of cut-off values and propoagative wavenumbers for the Navier equation.
\begin{definition} \label{def1}
(i) $\hat{\alpha}\in[-1/2,1/2]$ is called a cut-off value if there exists $\ell\in\mathbb{Z}$ that $|\hat{\alpha}+\ell|=k_{p}$ or $|\hat{\alpha}+\ell|=k_{s}$.
	
(ii) $\hat{\alpha}\in[-1/2,1/2]$ is called a propagative wave number if there exists a non-trivial $\phi\in [H_{{\rm loc}}^1(D)]^2$ such that
\begin{equation*}
(\Delta^*+\omega^2)\phi=0 ~~{\rm in }~~ D, \quad\quad \phi=0 ~~{\rm on }~~ \Lambda,
\end{equation*}
and $\phi$ satisfies the upward $\hat{\alpha}$-quasiperiodic Rayleigh expansion \eqref{a4} with $\alpha:=\hat{\alpha}$.
\end{definition}

In Definition \ref{def1} we restrict the quasi-periodic parameter $\hat{\alpha}$ to the interval $[-1/2,1/2]$, because an $\alpha$-quasi-periodic function must be also $(\alpha+n)$-quasi-periodic for any $n\in\mathbb{N}$. The cut-off values depends on the incident angle and frequency, while the propagative wavenumbers rely on both the incoming wave and the grating structure. The curves $\hat{\alpha}\mapsto \omega(\hat{\alpha})$ with $\hat{\alpha}\in[-1/2,1/2]$  constitute the so-called dispersion relations for the Navier equation in an open periodic waveguide. We rewrite the quasi-periodic momentum $\alpha$ given by \eqref{alpha} as
\be\label{ha}
\alpha=\ell+\hat{\alpha}\quad\mbox{for some}\quad \ell\in \Z\quad\mbox{and}\quad \hat{\alpha}\in[-1/2, 1/2].
\en
We call \( \alpha \in \mathbb{R} \) a cut-off value or a propagative wavenumber whenever \( \hat{\alpha} \) is of the same type.

The periodic variational form can be equivalently written as the operator equation (see \cite{EH2010})
\be\label{operator}
 (I-K_\omega)v=f_\omega, \quad v, f_\omega\in V_{\rm per}(\Omega_b),
\en
where $K_\omega$ is compact and depends analytically on $\omega$ such that $\hat{\alpha}=\hat{\alpha}(\omega)$ is not a cut-off value. Below we collect some well-posed results (see e.g. \cite{EH2010}):
\begin{itemize}
  \item[(i)] By the analytical Fredholm theory, the inverse operator $(I-K_\omega)^{-1}$ exists for all $\omega\in \R_+$ except for a countable number of frequencies with the only accumulating point at infinity.
 	
  \item[(ii)] If $\hat{\alpha}(\omega)$ is not a propagative wavenumber (that is, uniqueness holds true), $(I-K_\omega)^{-1}$ always exists by the Fredholm alternative.
 	
  \item[(iii)] If $\hat{\alpha}(\omega)$ happens to be a propagative wavenumber, there exists at least one solution to the variational formulation \eqref{var0}, because $F_\omega$ is always orthogonal to the kernel space $\mathcal{N}(I-K_\omega)^*$ of the adjoint operator.
 	
  \item[(iv)] If $\Lambda$ is a Lipschitz graph, then propagative wavenumbers does not exist and thus the operator $I-K_{\omega}$ is always invertible. If the grating profile $\Lambda$ is given by a general Lipschitz curve, one can prove well-posedness for sufficiently small frequencies $\omega$.
\end{itemize}

In the third case (iii), a general solution to $(I-K_\omega)v=f_\omega$ takes the form
\be\label{general}
v=v_0+\sum_{j=1}^m c_j\;\phi_j, \quad \phi_j\in \mathcal{N}(I-K_\omega),  \quad m=\mbox{dim}\, \mathcal{N}(I-K_\omega)<\infty,
\en
where $c_j\in \C$ are arbitrary constants and $v_0\in V_{\rm per}(\Omega_b)$ is a particular solution to $(I-K_\omega)v_0=f_\omega$. The purpose of this work is to investigate the limiting absorption principle in the resonance case: when the quasi-momentum $\alpha$ is a propagative number for some incidence frequency $\omega\in \R_+$, we perturb $\omega$ by a small imaginary part with $\omega+i\epsilon$ and  investigate the convergence of the unique solution $v_\epsilon$ of the periodic variational formulation as $\epsilon \rightarrow 0^+$. The limiting solution leads to a special solution to the original diffraction problem, which satisfies an additional constraint condition in addition to the Rayleigh expansion condition.
In the remaining part of this paper, we  make the following assumption on the incident angle and frequency.
\begin{assumption}\label{assumption}
It holds that $|\alpha+\ell|\neq k_{p}$ and $|\alpha+\ell|\neq k_{s}$ for every $\ell\in\mathbb{Z}$, where $\alpha\in \R$ is given by $\eqref{alpha}$. In other words, the number $\hat{\alpha}$ defined by \eqref{ha} is not a cut-off value.	
\end{assumption}

Under the Assumption \ref{assumption}, the null space
$\mathcal{N}(I-K_\omega)$ only consists of exponentially decaying surface waves; see Lemma \ref{Lem-p} (i) and
Lemma \ref{Lem-s} (i).

\section{Limiting Absorption Principle in the resonance case}\label{sec:LAP}
The aim of this section is to prove well-posedness of the elastic scattering throught the Limiting Absorption Principle (LAP) in the resonance case, that is, when $\alpha$ happens to be some propgative wavenumber.
The fundamental idea of LAP is to regularize an ill-posed stationary problem (e.g., the lossless operator equation \eqref{operator} when $\alpha(\omega)$ happens to be some propagative wavenumber) by incorporating a small, physically motivated absorption term, thereby ensuring well-posedness, and subsequently recovering a physically admissible solution via a limiting process as the absorption parameter tends to zero. In this work, we perturb $\omega$ by appending a small positive imaginary part, $\omega+i\epsilon$, and then consider the limiting solution as $\epsilon\rightarrow 0^+$. The convergence and stability of this limiting procedure is guaranteed by the following abstract result from functional analysis \cite{Kirsch2023}.

\begin{lemma} \label{LAP}
Let $I=(0,\epsilon_0)$ for some small $\epsilon_0>0$. Let $K(\epsilon)$ be compact operators from some Hilbert space $X$ into itself and $f(\epsilon) \in \mathcal{R}(L(\epsilon))$ for all $\epsilon\in[0, \epsilon_0)$ where $L(\epsilon):=I-K(\epsilon)$. Furthermore, let $L(\epsilon)$ be one-to-one (thus invertible) for all $\epsilon \in I$ and let $L(0)=I-K(0)$ have Riesz number one. Let $P: X\rightarrow \mathcal{N}:=N(L(0))$ be the projection onto the nullspace of $L(0)$ along the direct decomposition $X=\mathcal{N} \oplus \mathcal{R}$ where $\mathcal{R}=\mathcal{R}(L(0))$. Finally, let $f(\epsilon)$ and $K(\epsilon)$ be continuously differentiable functions in $\epsilon \in[0,\epsilon_0)$ and let $P L^{\prime}(0)|_{\mathcal{N}}$ be an isomorphism from $\mathcal{N}$ onto itself where $L^{\prime}(0)$ denotes the one-sided derivative of $L(\epsilon)$ at $\epsilon=0^{+}$.
Then the mapping $\epsilon \mapsto v(\epsilon):=[L(\epsilon)]^{-1} f(\epsilon)$ has a continuous extension to a mapping from $[0, \epsilon_0)$ into $X$. The limit $v(0)=\lim \limits_{\epsilon \rightarrow 0^{+}} v(\epsilon)$ is the unique solution of the system
\begin{equation}\label{limit}
L(0) v(0)=f(0), \quad\left[P L^{\prime}(0)\right] v(0)=P f^{\prime}(0)
\end{equation}
where $f^{\prime}(0)$ denotes the one-sided derivative of $f(\epsilon)$ at $\epsilon=0$. Moreover, there exist $\delta \in\left(0, \epsilon_0\right)$ and $c>0$ such that
\begin{equation*}
\left\|v\left(\epsilon_1\right)\right\|_X \leq c\left[\sup _{\epsilon \in[0, \delta]}\|f(\epsilon)\|_X+\sup _{\epsilon \in[0, \delta]}\left\|f^{\prime}(\epsilon)\right\|_X\right] \quad \text { for all } \epsilon_1 \in[0, \delta] .
\end{equation*}
\end{lemma}
The original version of Lemma \ref{LAP} dates back to
\cite[Theorem 1.32, Section 1.4.]{Colton2013}, which
was known as a singular perturbation theory.
The first equation in \eqref{limit} ensures that the limiting solution satisfies the original scattering problem, while the second one serves as an additional constraint that enforces uniqueness and stability. In the subsequent sections, we shall consider the incident pressure and shear plane waves separately; however, the two cases ultimately give rise to the same constraint condition.

\subsection{Incident pressure wave}\label{subsec:pressure}
We first consider the scattering problem with the incident pressure wave $u^{\text{in}}=u_p^{\text {in}}(x)=\hat{\theta} \exp \left(i k_p \hat{\theta}\cdot x\right)$. Set $\alpha:=k_p\sin\theta$. To consider frequency-independent energy spaces, we need to investigate the operator equation \eqref{operator} over the periodic Sobolev space $V_{\rm per}(\Omega_b)$. By the representation theorem of Riesz, there exist $f_{\omega} \in V_{{\rm per}}\left(\Omega_b\right)$ and a linear bounded operator $L_{\omega}$ from $V_{{\rm per}}\left(\Omega_b\right)$ into itself such that, for all $\psi\in  V_{{\rm per}}\left(\Omega_b\right)$ (see e.g. \eqref{R-p} for the first identity below),
\begin{align*}
\langle f_{\omega},\psi\rangle
=&2i\tau_{2}\omega e^{-ib\,\omega\,\cos\theta/\sqrt{\lambda+2\mu}} \int_{0}^{2\pi}  \left(\begin{array}{c}
	-\sin\theta \\  \tau_{1}\sqrt{\lambda+2\mu}
\end{array}\right)\cdot\overline{\psi(x_{1},b)}\, \mathrm{d}x_{1},
\end{align*}
\begin{equation*}
\tau_{1}:=\sqrt{\frac{\lambda+\mu(1+\cos^{2}\theta)}{\mu(\lambda+2\mu)}},  \quad  \tau_{2}:=\frac{\cos\theta \sqrt{\lambda+2\mu}}{\sin^{2}\theta+\tau_{1}\cos\theta \sqrt{\lambda+2\mu}},\quad
\end{equation*}
and
\begin{align*}
\langle L_{\omega}v,\psi\rangle=&B_{\omega}(v,\psi)=\int_{\Omega_b}\left(E(v, \overline{\psi})-(\omega^2-\mu\alpha^{2})v \cdot \overline{\psi}-2i\alpha\mu\frac{\partial v}{\partial x_{1}}\cdot\overline{\psi}\right) \mathrm{d}x-\int_{\Gamma_b} \overline{\psi} \cdot \mathscr{T}_{\omega}v \mathrm{~d}s \nonumber \\
&-i\alpha(\lambda+\mu)\int_{\Omega_b}\left[\left(\frac{\partial v_{2}}{\partial x_{2}}+2\frac{\partial v_{1}}{\partial x_{1}}\right)\overline{\psi}_{1}+ \frac{\partial v_{1}}{\partial x_{2}}\overline{\psi}_{2}\right] \mathrm{d}x +(\lambda+\mu)\int_{\Omega_b}\alpha^{2}v_{1}\overline{\psi}_{1}\mathrm{d}x \nonumber\\
=&\int_{\Omega_b}\left(E(v,\overline{\psi})-\mu\tau_{1}^{2}\omega^2v \cdot \overline{\psi}-\frac{2i\mu\sin\theta}{\sqrt{\lambda+2\mu}}\omega\frac{\partial v}{\partial x_{1}}\cdot\overline{\psi}\right) \mathrm{d}x \nonumber \\
&-\frac{i(\lambda+\mu)\sin\theta}{\sqrt{\lambda+2\mu}}\omega \int_{\Omega_b}\left[\left(\frac{\partial v_{2}}{\partial x_{2}}+2\frac{\partial v_{1}}{\partial x_{1}}\right)\overline{\psi}_{1}+ \frac{\partial v_{1}}{\partial x_{2}}\overline{\psi}_{2}\right] \mathrm{d}x  \nonumber\\
&+\frac{(\lambda+\mu)\sin^{2}\theta}{\lambda+2\mu}\omega^{2} \int_{\Omega_b}v_{1}\overline{\psi}_{1}\mathrm{d}x +2\pi\sum_{n\in\mathbb{Z}} \overline{\psi}_{n} \cdot \left(W_n(\omega) v_n\right).
\end{align*}
Choose $n_0>k_s>0$ sufficiently large such that the 2-by-2 matrix ${\rm Re}\,W_n$ is positive definite over $\C^2$ for all $|n|>n_0$; see \cite[Lemma 2 (i)]{EH2010} or \cite[Lemma 3.1 (i)]{WLX}. We then introduce the bounded operator $\mathscr{T}_{\omega,1}: H_{\rm per}^{1/2}(\Gamma_b)\rightarrow H_{\rm per}^{-1/2}(\Gamma_b)$ by
\ben
\mathscr{T}_{\omega,1}v=-\sum_{|n|\geq n_{0}} W_n(\omega) v_n e^{i n x_1},\quad v(x_1):=\sum_{n\in \Z}  v_n e^{i n x_1}\in H_{\rm per}^{1/2}(\Gamma_b).
\enn
Then it is obvious that $\mathscr{T}_{\omega,1}$ is positive definite over  $H_{\rm per}^{1/2}(\Gamma_b)$ and  $\mathscr{T}_{\omega}-\mathscr{T}_{\omega,1}$ is a finite dimensional operator.
We equip $V_{{\rm per}}\left(\Omega_b\right)$ with the inner product
\begin{equation}\label{product}
\langle v,\psi\rangle:=\int_{\Omega_b}E(v,\overline{\psi})\mathrm{~d}x-\int_{\Gamma_b} \overline{\psi} \cdot \mathscr{T}_{\omega,1}v \mathrm{~d} s.
\end{equation}
Referring to \cite[Remark 2]{EH2010}, one can always find a positive constant $C$ such that
\begin{equation*}
\int_{\Omega_b}E(v,\overline{v})\mathrm{~d}x \geq C\,|| v||_{V_{{\rm per}}\left(\Omega_b\right)}^{2}, \quad \forall \, v\in V_{{\rm per}}\left(\Omega_b\right).
\end{equation*}
Hence, the norm derived from the above inner product \eqref{product} is an equivalent norm of $V_{{\rm per}}\left(\Omega_b\right)$. Then the operator equation corresponding to the periodic variational formulation can be rewritten as
\begin{equation} \label{c0}
L_{\omega}v=f_{\omega} \quad {\rm in}~V_{{\rm per}}\left(\Omega_b\right).
\end{equation}
Using the compact embedding of $V_{{\rm per}}\left(\Omega_b\right)\rightarrow[L_{{\rm per}}^{2}(\Omega_{b})]^{2}$, one can show that the operator $K_{\omega}:=I-L_{\omega}$, given by
\begin{align*}
\langle K_{\omega}v,\psi\rangle=&\int_{\Omega_b}\left((\omega^2-\mu\alpha^{2})v \cdot \overline{\psi}+2i\alpha\mu\frac{\partial v}{\partial x_{1}}\cdot\overline{\psi}\right) \mathrm{d} x+\int_{\Gamma_b} \overline{\psi}\cdot (\mathscr{T}_{\omega}-\mathscr{T}_{\omega,1})v\mathrm{~d} s \\
&+i\alpha(\lambda+\mu)\int_{\Omega_b}\left[\left(\frac{\partial v_{2}}{\partial x_{2}}+2\frac{\partial v_{1}}{\partial x_{1}}\right)\overline{\psi}_{1}+ \frac{\partial v_{1}}{\partial x_{2}}\overline{\psi}_{2}\right] \mathrm{d}x -(\lambda+\mu)\int_{\Omega_b}\alpha^{2}v_{1}\overline{\psi}_{1}\mathrm{d}x,
\end{align*}
for all $v$, $\psi\in V_{{\rm per}}(\Omega_{b})$ is compact as an operator from $V_{{\rm per}}(\Omega_{b})$ into itself. Below we collect some properties of the operator $L_{\omega}$, which extends the results of \cite[Lemma 4.1]{HuKiZh} from the Helmholtz equation to the Navier equation.

\begin{lemma} \label{Lem-p}
Let $\alpha=k_{p}\sin\theta$, suppose that $|\alpha_{n}|\neq k_{p}$ and $|\alpha_{n}|\neq k_{s}$ for any $n\in\mathbb{Z}$.
	
(i) The null space $\mathcal{N}:=\mathcal{N}(L_{\omega})=\mathcal{N}(L_{\omega}^{*})$ is finite dimensional and consists of exponentially decaying modes.
	
(ii) The Riesz number of $L_{\omega}$ is one, that is, $\mathcal{N}(L_{\omega}) =\mathcal{N}(L_{\omega}^{2})$. Moreover, there holds the orthogonal decomposition $V_{{\rm per}}(\Omega_{b})=\mathcal{N}(L_{\omega})\oplus\mathcal{R}(L_{\omega})$. Here $\mathcal{R}(L_{\omega})$ denotes the range of the operator $L_{\omega}$.
	
(iii) If ${\rm Im}\,\omega>0$, there is a unique solution to \eqref{c0} for any $b> \Lambda^{+}$.
\end{lemma}
\begin{proof}
(i) For $v\in\mathcal{N}(L_{\omega})$, setting $\varphi=u=ve^{i\alpha x_{1}}$ in the homogeneous form of equation \eqref{var0}, we get
\begin{align*}
0=\widetilde{B}_{\omega}(u,u)=&\int_{\Omega_b}\left(E(u, \overline{u})-\omega^2 u \cdot \overline{u}\right) \mathrm{d} x-\int_{\Gamma_b} \overline{u} \cdot \mathscr{\widetilde{T}}_{\omega} u \mathrm{~d}s  \\
=&\int_{\Omega_b}\left(E(u, \overline{u})-\omega^2 |u|^{2}\right) \mathrm{d} x+2\pi\sum_{n\in\mathbb{Z}} \overline{u}_{n} \cdot \left(W_n(\omega) u_n\right).
\end{align*}
Taking the imaginary yields (refer to Lemma 4 in \cite{EH2010})
\begin{equation*}
0=\sum_{n\in\mathbb{Z}} \overline{u}_{n} \cdot \left[\,{\rm Im}W_n(\omega)\, u_n\,\right] ={\rm Im}\int_{\Gamma_b} \overline{u} \cdot \mathscr{\widetilde{T}}_{\omega} u \mathrm{~d}s =2\pi\omega^{2}\left( \sum_{|\alpha_{n}|<k_{p}}\beta_{n}|A_{p,n}|^{2} +\sum_{|\alpha_{n}|<k_{s}}\gamma_{n}|A_{s,n}|^{2}\right),
\end{equation*}
that is, $A_{p,n}=0$ for $|\alpha_{n}|<k_{p}$ and $A_{s,n}=0$ for $|\alpha_{n}|<k_{s}$. This together with the Assumption \ref{assumption} gives
\begin{equation*}
u(x)=\sum_{|\alpha_{n}|> k_{p}}A_{p,n}\left(\begin{array}{c}
	\alpha_n \\   \beta_n
\end{array}\right) \exp\left(i\alpha_n x_1+i\beta_n x_2\right) +\sum_{|\alpha_{n}|> k_{s}}A_{s, n}\left(\begin{array}{c}
	\gamma_n \\  -\alpha_n
	\end{array}\right) \exp \left(i \alpha_n x_1+i \gamma_n x_2\right),
\end{equation*}
that is, $u$ and also $v$ consist of exponentially decaying modes only. The kernal space $\mathcal{N}(L_\omega)$ must be finite dimensional, because $L_\omega=I-K_\omega$ is a Fredholm operator. This implies that
\begin{align} \label{Lw}
\langle L_{\omega}v,\psi\rangle=\int_{\Omega_b}\left(E(u, \overline{\varphi})-\omega^2 u\cdot \overline{\varphi}\right) \mathrm{d} x+2\pi\sum_{|\alpha_n|>k_p} \big(W_n(\omega) \,u_{n}, \overline{\varphi}_n \big)_{\C^2},
\end{align}
where $\big(\cdot, \cdot\big)_{\C^2}$ represents the inner product in $\C^2$.
Here $u=e^{i\alpha x_{1}}v$ and $\varphi=e^{i\alpha x_{1}}\psi$, and $u_n, \varphi_n$ denote the Fourier coefficients of $u$ and $\varphi$ (see \eqref{un}), respectively.

It remains to prove $\mathcal{N}(L_\omega)=\mathcal{N}(L^*_\omega) $.
The adjoint operator $L_{\omega}^{*}$ of $L_{\omega}$ is defined by
\begin{align*}
&\langle L_{\omega}^{*}v,\psi\rangle=\langle v,L_{\omega}\psi\rangle=\overline{\langle L_{\omega}\psi,v\rangle}=\overline{\widetilde{B}_{\omega}(\varphi,u)}  \\ =&\int_{\Omega_b}\left(E(u, \overline{\varphi})-\omega^2 u\cdot \overline{\varphi}\right) \mathrm{d} x+2\pi\sum_{n\in\mathbb{Z}}
\big(u_{n}, W_n(\omega) \varphi_n \big)_{\C^2}, \quad \forall v,\psi\in V_{{\rm per}}(\Omega_{b}).
\end{align*}
Arguing as in the above, it follows that each $v\in \mathcal{N}(L_{\omega}^{*})$ has vanishing Rayleigh coefficients of the incoming modes, that is $A_{p,n}=0$ for $|\alpha_{n}|<k_{p}$ and $A_{s,n}=0$ for $|\alpha_{n}|<k_{s}$.
Noting that
\ben
\big(u_{n}, W_n(\omega) \varphi_n \big)_{\C^2}=
\big( \overline{W_n(\omega)}^\top\,u_{n}, \varphi_n \big)_{\C^2}
\enn
where $()^\top$ denotes the transpose of a matrix, we get	
\begin{align} \label{Lw*}
\langle L_{\omega}^{*}v,\psi\rangle
=\int_{\Omega_b}\left(E(u, \overline{\varphi})-\omega^2 u\cdot \overline{\varphi}\right) \mathrm{d} x+2\pi\sum_{|\alpha_n|>k_p} \big( \overline{W_n(\omega)}^\top u_{n}, \varphi_n \big)_{\C^2} .
\end{align}
Using \eqref{un}, we have
\be\label{Wn}
W_n(\omega) \,u_{n}=\frac{1}{i}\left(\begin{array}{cc}
\omega^{2}\beta_n/d_{n} & 2\mu\alpha_{n}-\omega^{2}\alpha_{n}/d_{n} \\
-2\mu\alpha_{n}+\omega^{2}\alpha_{n}/d_{n} & \omega^{2}\gamma_n/d_{n}
\end{array}\right)\left(\begin{array}{cc}
\alpha_n & \gamma_n \\
\beta_n & -\alpha_n
\end{array}\right) \left(\begin{array}{c}
A_{p,n} \exp \left(i \beta_n b\right) \\
A_{s,n} \exp \left(i \gamma_n b\right)
\end{array}\right).
\en
  Case (i): $|\alpha_n|>k_s$. Since both $\beta_n$ and $\gamma_n$ are purely imaginary, it holds that $\overline{\beta_n}=-\beta_n$, $\overline{\gamma_n}=-\gamma_n$ and $\overline{d_n}=d_n=(\alpha_n^2+\beta_n\gamma_n)$. Hence
  \ben
  \overline{W_n(\omega)}^\top =
  \frac{-1}{i}\left(\begin{array}{cc}
  -\omega^{2}\beta_n/d_{n} & -2\mu\alpha_{n}+\omega^{2}\alpha_{n}/d_{n} \\
  2\mu\alpha_{n}-\omega^{2}\alpha_{n}/d_{n} & -\omega^{2}\gamma_n/d_{n}
  \end{array}\right)=W_n(\omega),
  \enn
  implying that $\overline{W_n(\omega)}^\top u_n=W_n(\omega)u_n$ for all $u_n \in \C^2$.
  
\noindent Case (ii): $k_p<|\alpha_n|<k_s$. In this case we use $A_{s,n}=0$ to rewrite the expression \eqref{Wn} as (see also \cite{EH2010})
  \ben
  W_n(\omega) \,u_{n}=\frac{1}{i}\left(\begin{array}{cc}
  2\mu\alpha_n \beta_n & \omega^2-2\mu \alpha_n^2 \\
  \omega^2-2\mu \alpha_n^2 & -2\mu\alpha_n \gamma_n
  \end{array}\right) \left(\begin{array}{c} A_{p,n} \exp \left(i \beta_n b\right) \\
  0
  \end{array}\right) 
  =-i \left(\begin{array}{c}
  2\mu\alpha_n \beta_n \\
  \omega^2-2\mu \alpha_n^2
  \end{array}\right)	A_{p,n} \exp \left(i \beta_n b\right).
  \enn
  Similarly, using $\overline{\beta_n}=-\beta_n$, $\overline{\gamma_n}=\gamma_n$ and $\overline{d_n}=(\alpha_n^2-\beta_n\gamma_n)$, we obtain
  \ben
  \overline{W_n(\omega)}^\top \,u_{n}  &=&  \frac{-1}{i}\left(\begin{array}{cc}
  -\omega^{2}\beta_n/\overline{d_{n}}&-2\mu\alpha_{n}+\omega^{2}\alpha_{n}/\overline{d_{n}}\\
  2\mu\alpha_{n}-\omega^{2}\alpha_{n}/\overline{d_{n}} & \omega^{2}\gamma_n/\overline{d_{n}}
  \end{array}\right)\left(\begin{array}{cc}
  \alpha_n \\  \beta_n
  \end{array}\right)  A_{p,n} \exp \left(i \beta_n b\right)\\
  &=&i \left(\begin{array}{c}
	-2\mu\alpha_{n}\beta_n \\
	2\mu \alpha_n^2-\omega^2
  \end{array}\right)	A_{p,n} \exp \left(i \beta_n b\right)=W_n(\omega) \,u_{n}.
  \enn
Now, combining \eqref{Lw} and \eqref{Lw*} yields that
\ben
\langle L_{\omega}v,\psi\rangle=0 \quad\mbox{if and only if}\quad \langle L^*_{\omega}v,\psi\rangle=0,
\enn
that is, $\mathcal{N}(L_\omega)=\mathcal{N}(L^*_\omega) $.	
	
(ii) It is obvious that $\mathcal{N}(L_{\omega})\subseteq\mathcal{N}(L_{\omega}^{2})$. To prove the reverse direction, we assume $L_{\omega}^{2}w=0$ for some $w\in V_{{\rm per}}(\Omega_{b})$ and set $v=L_{\omega}w\in\mathcal{R}(L_{\omega})$. Since $v\in\mathcal{N}(L_{\omega}) =\mathcal{N}(L_{\omega}^{*})$, we obtain
\begin{equation*}
\|v\|^{2}=\langle v,v\rangle=\langle v,L_{\omega}w\rangle=\langle L_{\omega}^{*}v,w\rangle=0,
\end{equation*}
which proves $\mathcal{N}(L_{\omega}^{2})\subseteq\mathcal{N}(L_{\omega})$ and thus the coincidence $\mathcal{N}(L_{\omega}) =\mathcal{N}(L_{\omega}^{2})$. This also implies $\mathcal{N}\cap\mathcal{R}(L_{\omega})=\emptyset$ and hence $V_{{\rm per}}(\Omega_{b}) =\mathcal{N}(L_{\omega})\oplus\mathcal{R}(L_{\omega})$. The orthogonality between $\mathcal{N}$ and $\mathcal{R}(L_{\omega})$ follows from the relation $\mathcal{N}(L_{\omega})= \mathcal{N}(L_{\omega}^{*})$.
	
(iii) We claim that ${\rm Im}\,\beta_{n}>0$, ${\rm Im}\, \gamma_{n}>0$ for all $n\in\mathbb{Z}$ if ${\rm Im}\,\omega>0$. Recall that the square root function $z \mapsto \sqrt{z}$ was chosen to be holomorphic in the region $\{z \in \mathbb{C}: z \notin i \mathbb{R}_{\leq 0}\}$. In particular we have that
\be\label{imag}
{\rm Im} \sqrt{z}>0\quad\mbox{for all}\quad z \in \mathbb{\widetilde{C}}:=\{z \in \mathbb{C}: {\rm Re}\, z<0 \quad\mbox{or}\quad {\rm Im}\, z>0\}.
\en
Denoting $\omega:=\omega_{0}+i\varepsilon$ with $\varepsilon>0$, then
\begin{align*}
\beta_{n}^{2}=&k_{p}^{2}-\alpha_{n}^{2}=\frac{\omega^{2}}{\lambda+2\mu} -\left(n+\frac{\omega}{\sqrt{\lambda+2\mu}}\sin\theta\right)^{2} \\
=&\frac{\omega_{0}^{2}-\varepsilon^{2}}{\lambda+2\mu} -\left(n+\frac{\omega_{0}\sin\theta}{\sqrt{\lambda+2\mu}}\right)^{2} +\frac{\varepsilon^{2}\sin^{2}\theta}{\lambda+2\mu} +\frac{2i\varepsilon\omega_{0}}{\lambda+2\mu} -\frac{2i\varepsilon\sin\theta}{\sqrt{\lambda+2\mu}} \left(n+\frac{\omega_{0}\sin\theta}{\sqrt{\lambda+2\mu}}\right),
\end{align*}
and
\begin{align*}
\gamma_{n}^{2}=&k_{s}^{2}-\alpha_{n}^{2}=\frac{\omega^{2}}{\mu} -\left(n+\frac{\omega}{\sqrt{\lambda+2\mu}}\sin\theta\right)^{2} \\
=&\frac{\omega_{0}^{2}-\varepsilon^{2}}{\mu} -\left(n+\frac{\omega_{0}\sin\theta}{\sqrt{\lambda+2\mu}}\right)^{2} +\frac{\varepsilon^{2}\sin^{2}\theta}{\lambda+2\mu} +\frac{2i\varepsilon\omega_{0}}{\mu} -\frac{2i\varepsilon\sin\theta}{\sqrt{\lambda+2\mu}} \left(n+\frac{\omega_{0}\sin\theta}{\sqrt{\lambda+2\mu}}\right).
\end{align*}
We compute the real and imaginary parts of $\beta_n^2$ and $\gamma_n^2$ as follows:
\ben
{\rm Re}\beta_{n}^{2}&=&\frac{\omega_{0}^{2}-\varepsilon^{2}}{\lambda+2\mu} -\left(n+\frac{\omega_{0}\sin\theta}{\sqrt{\lambda+2\mu}}\right)^{2} +\frac{\varepsilon^{2}\sin^{2}\theta}{\lambda+2\mu}, \quad
{\rm Im}\beta_{n}^{2}=\frac{2\varepsilon\omega_{0}}{\lambda+2\mu} -\frac{2\varepsilon\sin\theta}{\sqrt{\lambda+2\mu}} \left(n+\frac{\omega_{0}\sin\theta}{\sqrt{\lambda+2\mu}}\right),\\
{\rm Re}\gamma_{n}^{2}&=&\frac{\omega_{0}^{2}-\varepsilon^{2}}{\mu} -\left(n+\frac{\omega_{0}\sin\theta}{\sqrt{\lambda+2\mu}}\right)^{2} +\frac{\varepsilon^{2}\sin^{2}\theta}{\lambda+2\mu}, \quad
{\rm Im}\gamma_{n}^{2}=\frac{2\varepsilon\omega_{0}}{\mu} -\frac{2\varepsilon\sin\theta}{\sqrt{\lambda+2\mu}} \left(n+\frac{\omega_{0}\sin\theta}{\sqrt{\lambda+2\mu}}\right).
\enn
	
Assume that ${\rm Im}\beta_{n}^{2}\leq0$, then we have
\begin{equation*}
0<\frac{\varepsilon\omega_{0}}{\lambda+2\mu} \leq \frac{\varepsilon\sin\theta}{\sqrt{\lambda+2\mu}} \left(n+\frac{\omega_{0}\sin\theta}{\sqrt{\lambda+2\mu}}\right),
\end{equation*}
which implies that
\begin{equation*}
{\rm Re}\beta_{n}^{2}\leq \frac{\omega_{0}^{2}-\varepsilon^{2}\cos^{2}\theta}{\lambda+2\mu} -\frac{\omega_{0}^{2}}{(\lambda+2\mu)\sin^{2}\theta} =\frac{-1}{\lambda+2\mu}\left[\varepsilon^{2}\cos^{2}\theta + \left(\frac{1}{\sin^{2}\theta}-1\right)\omega_{0}^{2}\right]<0.
\end{equation*}
	
Assume that ${\rm Im}\gamma_{n}^{2}\leq0$, then we have
\begin{equation*}
0<\frac{\varepsilon\omega_{0}}{\mu} \leq\frac{\varepsilon\sin\theta}{\sqrt{\lambda+2\mu}} \left(n+\frac{\omega_{0}\sin\theta}{\sqrt{\lambda+2\mu}}\right),
\end{equation*}
which implies that
\begin{equation*}
{\rm Re}\gamma_{n}^{2}\leq\frac{\omega_{0}^{2}-\varepsilon^{2}}{\mu} +\frac{\varepsilon^{2}\sin^{2}\theta}{\lambda+2\mu} -\frac{(\lambda+2\mu)\omega_{0}^{2}}{\mu^{2}\sin^{2}\theta} =\varepsilon^{2}\left(\frac{\sin^{2}\theta}{\lambda+2\mu}-\frac{1}{\mu}\right) +\frac{\omega_{0}^{2}}{\mu}\left(1-\frac{\lambda+2\mu}{\mu\sin^{2}\theta}\right) <0.
\end{equation*}
Using \eqref{imag}, we conclude that ${\rm Im}\,\beta_{n}>0$, ${\rm Im}\, \gamma_{n}>0$ for all $n\in\mathbb{Z}$.
	
Let $v\in V_{{\rm per}}(\Omega_{\infty})$ be a solution of the homogeneous problem \eqref{a0}, \eqref{a1}, \eqref{a4} with ${\rm Im}\,\omega>0$. By the Rayleigh expansion \eqref{a4} and the first assertion, $v$ must exponetially decay in the $x_{2}$-direction, allowing a variational formulation over $V_{{\rm per}}(\Omega_\infty)$:
\begin{align*}
&\int_{\Omega_\infty}\left(E(v,\overline{\psi})-\mu\tau_{1}^{2}\omega^2v \cdot \overline{\psi}-\frac{2i\mu\sin\theta}{\sqrt{\lambda+2\mu}}\omega\frac{\partial v}{\partial x_{1}}\cdot\overline{\psi}\right) \mathrm{d}x +\frac{(\lambda+\mu)\sin^{2}\theta}{\lambda+2\mu}\omega^{2} \int_{\Omega_\infty}v_{1}\overline{\psi}_{1}\mathrm{d}x  \\
		&-\frac{i(\lambda+\mu)\sin\theta}{\sqrt{\lambda+2\mu}}\omega \int_{\Omega_\infty}\left[\left(\frac{\partial v_{2}}{\partial x_{2}}+2\frac{\partial v_{1}}{\partial x_{1}}\right)\overline{\psi}_{1}+ \frac{\partial v_{1}}{\partial x_{2}}\overline{\psi}_{2}\right] \mathrm{d}x=0, \quad \forall \, \psi\in V_{{\rm per}}(\Omega_\infty),
\end{align*}
where $\Omega_\infty=\{x\in D, 0<x_{1}<2\pi\}$. We rewrite the above sesquilinear form as a quadratic operator equation:
\begin{equation*}
\mathbf{A}v-\omega \mathbf{B}v-\omega^{2}\mathbf{C}v=0, \quad \forall \, \psi\in V_{{\rm per}}(\Omega_\infty),
\end{equation*}
where the $\omega$-independent operators $\mathbf{A}$, $\mathbf{B}$ and $\mathbf{C}$ are defined by
\be \label{A1}
	\langle \mathbf{A}v,\psi\rangle&=&\int_{\Omega_\infty}E(v,\overline{\psi})\,\mathrm{d}x,\\ \label{B1}
	\langle \mathbf{B}v,\psi\rangle &=&\frac{2i\mu\sin\theta}{\sqrt{\lambda+2\mu}}\int_{\Omega_\infty}\frac{\partial v}{\partial x_{1}}\cdot\overline{\psi}\, \mathrm{d}x +\frac{i(\lambda+\mu)\sin\theta}{\sqrt{\lambda+2\mu}} \int_{\Omega_\infty}\left[\left(\frac{\partial v_{2}}{\partial x_{2}}+2\frac{\partial v_{1}}{\partial x_{1}}\right)\overline{\psi}_{1}+ \frac{\partial v_{1}}{\partial x_{2}}\overline{\psi}_{2}\right] \mathrm{d}x,\\
 \label{C1}
	\langle \mathbf{C}v,\psi\rangle&=&\mu\tau_{1}^{2}\int_{\Omega_\infty}v \cdot \overline{\psi}\,\mathrm{d}x-\frac{(\lambda+\mu)\sin^{2}\theta}{\lambda+2\mu} \int_{\Omega_\infty}v_{1}\overline{\psi}_{1}\mathrm{d}x.
\en
It is easy to prove that $\mathbf{A}$ and $\mathbf{C}$ are self-adjoint bounded operators from $V_{{\rm per}}(\Omega_\infty)$ into itself. By integrating by parts,  for any $b>\Lambda^{+}$ we see
\begin{equation*}
\int_{\Omega_b}\frac{\partial v_{1}}{\partial x_{1}}\overline{\psi}_{1} +\frac{\partial \overline{\psi}_{1}}{\partial x_{1}}v_{1} \mathrm{d}x=\int_{\Omega_b}\frac{\partial v_{2}}{\partial x_{1}}\overline{\psi}_{2} +\frac{\partial \overline{\psi}_{2}}{\partial x_{1}}v_{2}\, \mathrm{d}x=0,
\end{equation*}
and
\begin{equation*}
\int_{\Omega_b}\left(\frac{\partial v_{2}}{\partial x_{2}}\overline{\psi}_{1} +\frac{\partial \overline{\psi}_{1}}{\partial x_{2}}v_{2}\right) +\left(\frac{\partial v_{1}}{\partial x_{2}}\overline{\psi}_{2} +\frac{\partial \overline{\psi}_{2}}{\partial x_{2}}v_{1}\right)\mathrm{d}x=0.
\end{equation*}
Taking $b\rightarrow\infty$, we conclude that $\mathbf{B}$ is also a self-adjoint bounded operator form $V_{{\rm per}}(\Omega_\infty)$ into itself. Note that $\mathbf{A}$ is positive, that is,
\begin{equation*}
\langle \mathbf{A}v, v\rangle=\int_{\Omega_\infty}E(v,\overline{v})\,\mathrm{d}x =\int_{\Omega_\infty} \left[\,\lambda|\nabla\cdot v|^{2}+2\mu|\mathcal{E}(v)|^{2}\, \right]\mathrm{d}x \geq \int_{\Omega_\infty}(\lambda+\mu)|\nabla\cdot v|^{2}\mathrm{d}x >0
\end{equation*}
for all $v \neq 0$, where $\mathcal{E}(v):=[\nabla v+(\nabla v)^{\top}]/2$. Moreover,
\begin{align*}
\langle \mathbf{C}v,v\rangle=&\mu\tau_{1}^{2}\int_{\Omega_\infty}|v|^{2}\mathrm{d}x -\frac{(\lambda+\mu)\sin^{2}\theta}{\lambda+2\mu} \int_{\Omega_\infty}|v_{1}|^{2}\mathrm{d}x \\
	=&\left(\mu\tau_{1}^{2}-\frac{(\lambda+\mu)\sin^{2}\theta}{\lambda+2\mu}\right) \int_{\Omega_\infty}|v_{1}|^{2}\mathrm{d}x +\mu\tau_{1}^{2}\int_{\Omega_\infty}|v_{2}|^{2}\mathrm{d}x \\
	=&\cos^{2}\theta \int_{\Omega_\infty}|v_{1}|^{2}\mathrm{d}x +\mu\tau_{1}^{2}\int_{\Omega_\infty}|v_{2}|^{2}\mathrm{d}x>0, \quad \forall v\neq0.
\end{align*}
Hence, $\mathbf{C}$ is also positive. Now we use that $\mathbf{C}$ has a square root $\mathbf{W}$; that is, a bounded and self-adjoint operator with $\mathbf{W}^2=\mathbf{C}$. Then we have $\mathbf{A}v-\omega \mathbf{B} u-\omega^2 \mathbf{W}^2 v=0$. Setting $V=(v,\omega \mathbf{W}v)^{\top}$, we obtain the equivalent system $\mathcal{A}V=\omega\mathcal{B}V$, where
\begin{equation*}
\mathcal{A}:=\left(\begin{array}{cc}
		\mathbf{A} & 0 \\  0 & I
\end{array}\right), \quad \mathcal{B}:=\left(\begin{array}{cc}
	\mathbf{B} & \mathbf{W} \\  \mathbf{W} & 0
\end{array}\right).
\end{equation*}
The operator matrices $\mathcal{A}$ and $\mathcal{B}$ are both self-adjoint in $V_{{\rm per }}(\Omega_\infty)\times V_{{\rm per}}(\Omega_\infty)$. If $V\neq 0$, then there holds
\begin{equation*}
\omega\langle\mathcal{B} V, V\rangle=\langle\mathcal{A} V, V\rangle
\end{equation*}
which shows that $\omega$ is real valued, contradicting the assumption that ${\rm Im}\, \omega >0$. Hence $V=0$, which gives $v=0$. This proves uniqueness for a complex-valued frequency with a positive imaginary part.
\end{proof}

To apply the singular perturbation argument from Lemma~\ref{LAP}, we set
\( X = V_{\mathrm{per}}(\Omega_b) \) and denote by
\( P : X \to \mathcal{N} \) the projection operator.
Fix the incident angle, and choose \( \omega \in \mathbb{R}^+ \) such that
\( \hat{\alpha}(\omega) \) is not a cut-off value. Consequently, the matrices
\( W_n(\omega) \) are analytic in a neighborhood of \( \omega \).
We now perturb \( \omega \) by introducing a small positive imaginary part, that is,
 \( \omega + i\epsilon \) with \( \epsilon > 0 \).
For notational convenience, we indicate the dependence on \( \epsilon \) by
setting \( f_{\omega+i\epsilon} = f(\epsilon) \) and \( L_{\omega+i\epsilon} = L(\epsilon) \).
With these conventions, for every \( \psi \in X \), we have

\begin{equation} \label{f}
\langle f(\epsilon),\psi\rangle=2i\tau_{2}(\omega+i\epsilon) e^{-ib(\omega+i\epsilon)\cos\theta/\sqrt{\lambda+2\mu}} \int_{0}^{2\pi} \left(\begin{array}{c}
	-\sin\theta \\  \tau_{1}\sqrt{\lambda+2\mu}
\end{array}\right)\cdot\overline{\psi(x_{1},b)}\,\mathrm{d}x_{1},
\end{equation}
\begin{align}  \label{L}
\langle L(\epsilon)v,\psi\rangle =&\int_{\Omega_b}\left(E(v,\overline{\psi})-\mu\tau_{1}^{2}(\omega+i\epsilon)^2v \cdot \overline{\psi}-\frac{2i\mu\sin\theta}{\sqrt{\lambda+2\mu}}(\omega+i\epsilon)\frac{\partial v}{\partial x_{1}}\cdot\overline{\psi}\right) \mathrm{d}x \nonumber \\
	&-\frac{i(\lambda+\mu)\sin\theta}{\sqrt{\lambda+2\mu}}(\omega+i\epsilon) \int_{\Omega_b}\left[\left(\frac{\partial v_{2}}{\partial x_{2}}+2\frac{\partial v_{1}}{\partial x_{1}}\right)\overline{\psi}_{1}+ \frac{\partial v_{1}}{\partial x_{2}}\overline{\psi}_{2}\right] \mathrm{d}x  \nonumber\\
	&+\frac{(\lambda+\mu)\sin^{2}\theta}{\lambda+2\mu}(\omega+i\epsilon)^{2} \int_{\Omega_b}v_{1}\overline{\psi}_{1}\mathrm{d}x +2\pi\sum_{n\in\mathbb{Z}} \overline{\psi}_{n} \cdot \left(W_n(\omega+i\epsilon)\, v_n\right).
\end{align}
From the above expressions we observe that $f(\epsilon)$ and $L(\epsilon)$ are differentiable with respect to $\epsilon$ in a neighborhood of $z=0$.

To apply Lemma~\ref{LAP}, we need to prove the following results.
\begin{lemma}\label{Lem3}
(i) $f(\epsilon)\in\mathcal{R}(L(\epsilon))$ for all $\epsilon>0$ and $F(0)$, $F'(0)\in \mathcal{R}(L(0))$.
	
(ii) $PL'(0)$ is one-to-one on $\mathcal{N}(L(0))$.
\end{lemma}
\begin{proof} (i)
By the definitions of $f(\epsilon)$ and $L(\epsilon)$ in \eqref{f}-\eqref{L}, it follows that
\begin{equation*}
\langle f'(\epsilon),\psi\rangle =2i\tau_{2}\left[i+\frac{\omega+i\epsilon}{\sqrt{\lambda+2\mu}}b\cos\theta\right] e^{-ib(\omega+i\epsilon)\cos\theta/\sqrt{\lambda+2\mu}} \int_{0}^{2\pi} \left(\begin{array}{c}
	-\sin\theta \\  \tau_{1}\sqrt{\lambda+2\mu}
\end{array}\right)\cdot\overline{\psi(x_{1},b)}\,\mathrm{d}x_{1},
\end{equation*}
and
\begin{align*}
\langle L'(\epsilon)v,\psi\rangle =&\int_{\Omega_b}\left(-2i\mu\tau_{1}^{2}(\omega+i\epsilon)v \cdot \overline{\psi}+\frac{2\mu\sin\theta}{\sqrt{\lambda+2\mu}}\frac{\partial v}{\partial x_{1}}\cdot\overline{\psi}\right) \mathrm{d}x \\	&+\frac{(\lambda+\mu)\sin\theta}{\sqrt{\lambda+2\mu}}\int_{\Omega_b}\left[\left(\frac{\partial v_{2}}{\partial x_{2}}+2\frac{\partial v_{1}}{\partial x_{1}}\right)\overline{\psi}_{1}+ \frac{\partial v_{1}}{\partial x_{2}}\overline{\psi}_{2}\right] \mathrm{d}x  \\
	&+2i\frac{(\lambda+\mu)\sin^{2}\theta}{\lambda+2\mu}(\omega+i\epsilon) \int_{\Omega_b}v_{1}\overline{\psi}_{1}\mathrm{d}x +2\pi i\sum_{n\in\mathbb{Z}} \overline{\psi}_{n} \cdot \left(W'_n(\omega+i\epsilon) v_n\right).
\end{align*}
			
(i) Recall from the periodic operator equation \eqref{c0} that $L(\epsilon)$ is the sum of the identity operator and a compact operator.
Together with the uniqueness of Lemma \ref{Lem-p} (iii), this implies that $L(\epsilon)$ is invertible and thus $f(\epsilon)\in \mathcal{R}(L(\epsilon))$ for all $\epsilon>0$. On the other hand, we have $f(0)\in \mathcal{R}(L(0))$, because by Lemma \ref{Lem-p} (i) and (ii), $f(0)$ is orthogonal to $\mathcal{N}(L(0))$ and $V_{{\rm per}}(\Omega_b)$ admits the orthogonal decomposition $V_{{\rm per}}(\Omega_b)=\mathcal{N}(L(0)) \oplus \mathcal{R}(L(0))$.
			
Since the null space $\mathcal{N}(L(0))$ consists of evanescent wave modes only and $(I-P)\phi$ is orthogonal to $\mathcal{N}(L(0))$, it holds that
\begin{equation*}
\langle Pf(0),\psi\rangle=\langle f(0),\psi\rangle=2i\tau_{2}\omega e^{-ib\,\omega\,\cos\theta/\sqrt{\lambda+2\mu}} \int_{0}^{2\pi} \left(\begin{array}{c}
			-\sin\theta \\  \tau_{1}\sqrt{\lambda+2\mu}
\end{array}\right)\cdot\overline{\psi(x_{1},b)}\,\mathrm{d}x_{1}=0,
\end{equation*}
and
\begin{align*}
&\langle Pf'(0),\psi\rangle =\langle f'(0),\psi\rangle  \\ =&2i\tau_{2}\left[i+\frac{\omega}{\sqrt{\lambda+2\mu}}b\cos\theta\right] e^{-ib\,\omega\,\cos\theta/\sqrt{\lambda+2\mu}} \int_{0}^{2\pi} \left(\begin{array}{c}
	-\sin\theta \\  \tau_{1}\sqrt{\lambda+2\mu}
\end{array}\right)\cdot\overline{\psi(x_{1},b)}\,\mathrm{d}x_{1}=0,
\end{align*}
for all $\psi\in \mathcal{N}(L(0))$. Here we have used the fact that $\int_{0}^{2\pi} \overline{\psi(x_{1},b)}\,\mathrm{d}x_{1}=0$ due to Lemma \ref{Lem-p} (ii).
This further implies that $f(0), f'(0)\in\mathcal{R}(L(0))$ and $Pf(0)=Pf'(0)=0$.
			
(ii) Assuming that $\langle PL'(0)v,\cdot\rangle$ vanishes identically on $\mathcal{N}(L(0))$ for some $v\in\mathcal{N}(L(0))$, we need to show that $v\equiv 0$.
			
We extend $v,\psi\in \mathcal{N}(L(0))$ from $\Omega_b$ to $\Omega_\infty$ by the Rayleigh expansion, which we still denote by $v$ and $\psi$.  Since both $v$ and $\psi$ exponentially decay as $x_2\rightarrow \infty$,   simple calculations show that
\be \nonumber	
0&=&\langle PL'(0)v,\psi\rangle=\langle L'(0)v,\psi\rangle  \\ \nonumber &=&\int_{\Omega_\infty}\left(-2i\mu\tau_{1}^{2}\omega v \cdot \overline{\psi}+\frac{2\mu\sin\theta}{\sqrt{\lambda+2\mu}}\frac{\partial v}{\partial x_{1}}\cdot\overline{\psi}\right) \mathrm{d}x +2i\frac{(\lambda+\mu)\sin^{2}\theta}{\lambda+2\mu}\omega \int_{\Omega_\infty}v_{1}\overline{\psi}_{1}\mathrm{d}x \\ \label{PL}
	&&+\frac{(\lambda+\mu)\sin\theta}{\sqrt{\lambda+2\mu}}\int_{\Omega_\infty} \left[\left(\frac{\partial v_{2}}{\partial x_{2}}+2\frac{\partial v_{1}}{\partial x_{1}}\right)\overline{\psi}_{1}+ \frac{\partial v_{1}}{\partial x_{2}}\overline{\psi}_{2}\right] \mathrm{d}x,
\en
for all $v,\psi\in \mathcal{N}(L(0))$.
Choosing $\psi=v$, we obtain
\begin{align*}
\langle PL'(0)v,v\rangle=&\int_{\Omega_\infty}\left(-2i\mu\tau_{1}^{2}\omega |v|^{2}+\frac{2\mu\sin\theta}{\sqrt{\lambda+2\mu}}\frac{\partial v}{\partial x_{1}}\cdot\overline{v}\right) \mathrm{d}x+2i\frac{(\lambda+\mu)\sin^{2}\theta}{\lambda+2\mu}\omega \int_{\Omega_\infty}|v_{1}|^{2}\mathrm{d}x  \\
	&+\frac{(\lambda+\mu)\sin\theta}{\sqrt{\lambda+2\mu}}\int_{\Omega_\infty} \left[\left(\frac{\partial v_{2}}{\partial x_{2}}+2\frac{\partial v_{1}}{\partial x_{1}}\right)\overline{v}_{1}+ \frac{\partial v_{1}}{\partial x_{2}}\overline{v}_{2}\right] \mathrm{d}x=0.
\end{align*}
Recalling the definitions of $\mathbf{A}$, $\mathbf{B}$ and $\mathbf{C}$ in \eqref{A1}-\eqref{C1}, the above relation can be rewritten as
\begin{equation*}
	\langle \mathbf{B}v,v\rangle+2\omega\langle \mathbf{C}v,v\rangle=0.
\end{equation*}
Since $v\in\mathcal{N}(L(0))$ is a solution of the homogeneous problem, we have $\mathbf{A}v- \omega \mathbf{B}v-\omega^{2}\mathbf{C}v=0$. Hence,
\begin{equation*}
	\langle \mathbf{A}v,v\rangle+\omega^{2}\langle \mathbf{C}v,v\rangle=0.
\end{equation*}
Observing that $\mathbf{A}$ and $\mathbf{C}$ are both positive operators, we obtain that $v=0$, proving that $PL'(0)$ is one-to-one on $\mathcal{N}(L(0))$.
\end{proof}
Lemmas \ref{Lem-p}-\ref{Lem3} enable us to apply the singular perturbation Lemma \ref{LAP}. Consequently,  the unique solution $v(\epsilon)$ to \eqref{c0} converges to $v$ in $X$ and the limiting function $v$ fulfills the periodic operator equations
\begin{equation*}
	L_{\omega}v=f_{\omega} \quad {\rm and} \quad  PL'(0)v=0.
\end{equation*}
It is the second equation which provides an additional constraint on $v\in X$ to ensure uniqueness if $\hat{\alpha}(\omega)$ happens to be some propagative wavenumber. This equation can be equivalently written as the following integral identity (cf. \eqref{PL}):
\begin{align} \label{v}
&\int_{\Omega_\infty}\left(-2i\mu\tau_{1}^{2}\omega v \cdot \overline{\psi}+\frac{2\mu\sin\theta}{\sqrt{\lambda+2\mu}}\frac{\partial v}{\partial x_{1}}\cdot\overline{\psi}\right) \mathrm{d}x +2i\frac{(\lambda+\mu)\sin^{2}\theta}{\lambda+2\mu}\omega \int_{\Omega_\infty}v_{1}\overline{\psi}_{1}\mathrm{d}x \nonumber\\
	&+\frac{(\lambda+\mu)\sin\theta}{\sqrt{\lambda+2\mu}}\int_{\Omega_\infty} \left[\left(\frac{\partial v_{2}}{\partial x_{2}}+2\frac{\partial v_{1}}{\partial x_{1}}\right)\overline{\psi}_{1}+ \frac{\partial v_{1}}{\partial x_{2}}\overline{\psi}_{2}\right] \mathrm{d}x=0, \quad \forall\, \psi\in \mathcal{N}.
\end{align}
Setting $u=e^{ik_{p}\sin\theta x_{1}}v$ and $\varphi=e^{ik_{p}\sin\theta x_{1}}\psi$, we return to quasi-periodic settings to get
\begin{equation}\label{orth}
	\frac{\sin\theta}{\sqrt{\lambda+2\mu}} \int_{\Omega_\infty}\left\{2\mu\frac{\partial u}{\partial x_{1}}\cdot\overline{\varphi}+(\lambda+\mu)\left[\left(\frac{\partial u_{2}}{\partial x_{2}}+2\frac{\partial u_{1}}{\partial x_{1}}\right)\overline{\varphi}_{1}+ \frac{\partial u_{1}}{\partial x_{2}}\overline{\varphi}_{2}\right] \right\} \mathrm{d}x =2i\omega\int_{\Omega_\infty}u\cdot\overline{\varphi}\,\mathrm{d}x,
\end{equation}
for all $\varphi\in[H_{{\rm loc}, \alpha,0}^{1}(D) ]^{2}$ such that  $\phi(x)\, e^{-ik_{p}\sin\theta x_{1}} \in \mathcal{N}(L_\omega)$.
\begin{remark}
(i) The improper integrals  on both sides of \eqref{orth} make sense, because by Lemma \ref{Lem-p}, $\psi\in \mathcal{N}(L_\omega)$ consists of exponentially decaying modes.

(ii) In the case where \( \alpha = k_p \sin\theta \) is not a propagative wavenumber, the set \( \mathcal{N}(L_\omega) \) is necessarily empty, so the aforementioned constraint condition \eqref{orth} becomes superfluous.
\end{remark}
The principal findings of this section are summarised below. Define
\begin{equation*}
[H_{{\rm loc},\alpha,0}^1(D)]^2:=\left\{u\in [H_{{\rm loc}}^{1}(D)]^{2}\, :\, u \text{ is } \alpha\text{-quasiperiodic}, \, u=0\, \text {on}\, \partial D \right\}.
\end{equation*}
		
\begin{theorem}\label{THP}
Let $\omega>0$ be fixed and $\theta\in(-\pi/2,\pi/2)$ be an arbitrary incident direction. Set $\alpha=k_{p}\sin\theta$ and suppose that $|\alpha+n|\neq k_{p}$, $|\alpha+n|\neq k_{s}$ for any $n\in\mathbb{Z}$. There exists a unique solution $u\in [H_{{\rm loc},\alpha,0}^1(D)]^2$ such that $u^{\rm sc}:=u-u^{\rm in}$ satisfies the $\alpha$-quasiperiodic Rayleigh expansion \eqref{a4} if $\alpha$ is not a propagative wave number. At a propagative wavenumber, there still exists a unique solution if we additionally enforce the following constraint condition
\begin{equation}\label{integral-p}
	\alpha \int_{\Omega_\infty}\left\{2\mu\frac{\partial u}{\partial x_{1}}\cdot\overline{\varphi}+(\lambda+\mu)\left[\left(\frac{\partial u_{2}}{\partial x_{2}}+2\frac{\partial u_{1}}{\partial x_{1}}\right)\overline{\varphi}_{1}+ \frac{\partial u_{1}}{\partial x_{2}}\overline{\varphi}_{2}\right] \right\} \mathrm{d}x =2i\omega^2\int_{\Omega_\infty}u\cdot\overline{\varphi}\,\mathrm{d}x,
\end{equation}
for all $\varphi\in[H_{{\rm loc}, \alpha,0}^{1}(D) ]^{2}$ such that $e^{-ik_{p}\sin\theta\, x_{1}}\varphi \in \mathcal{N}$.
\end{theorem}
\begin{proof}
By Lemma \ref{Lem-p} (i), existence of $u \in [H_{{\rm loc}, \alpha,0}^{1}(D) ]^{2}$ follows from the Fredholm alternative and uniqueness holds true if $\alpha$ is not a propagative wave number. If $\alpha$ is a propagative wavenumber, we assume there are two solutions $u_{1}$ and $u_{2}$. Set $w=u_{1}-u_{2}$. It then follows from the limiting absorption argument that the periodic function $v(x)=e^{-i\alpha\cdot x}w(x)\in\mathcal{N}$ fulfills the relation \eqref{v}, that is, $\langle PL^{\prime}(0)v, \psi\rangle=0$ for all $\psi\in\mathcal{N}$. Applying Lemma \ref{Lem3} (ii) yields $v=0$ and thus $w(x)=v(x)e^{i\alpha\cdot x}=0$.
\end{proof}
		
\subsection{Incident shear wave}\label{subsec:shear}
		
In this part, we consider the incident shear plane wave of the form $u^{\text{in}}=u_s^{\text{in}}(x) =\hat{\theta}^{\perp} \exp \left(i k_s \hat{\theta} \cdot x\right)$. Set $\alpha=k_s\sin\theta$ and suppose that the incident angle satisfies the relation
\begin{equation}\label{assu}
	\sin^{2}\theta<\frac{\mu}{\lambda+2\mu}.
\end{equation}
Under this assumption, we have
\ben		\beta_0=\sqrt{k_p^2-\alpha^2}=\omega\sqrt{\frac{1}{\lambda+2\mu}-\frac{1}{\mu}\sin^2\theta} >0.
\enn
By the representation theorem of Riesz, there exist $f_{\omega}\in V_{{\rm per}}(\Omega_b)$ such that
\begin{equation*}
\langle f_{\omega},\psi\rangle=-\frac{2i\gamma_0 k_s \mu}{d_0}e^{-i\gamma_0 b} \int_{0}^{2\pi}  \left(\begin{array}{c}
 \beta_0 \\  \alpha	\end{array}\right)\cdot\overline{\psi(x_{1},b)}\,\mathrm{d}x_{1} 
=-2i\tau_{4}\omega e^{-ib\omega\cos\theta/\sqrt{\mu}} \int_{0}^{2\pi}  \left(\begin{array}{c}
 \tau_{3} \\  \sin\theta	\end{array}\right)\cdot\overline{\psi(x_{1},b)}\,\mathrm{d}x_{1},
\end{equation*}
with
\begin{equation}\label{tau34}
\tau_{3}:=\sqrt{\frac{\mu-(\lambda+2\mu)\sin^{2}\theta}{\lambda+2\mu}}>0, \quad \tau_{4}:=\frac{\sqrt{\mu}\cos\theta}{\sin^{2}\theta+\tau_{3}\cos\theta}>0.
\end{equation}
Analogously, one can find a linear bounded operator $L_{\omega}$ from $V_{{\rm per}}(\Omega_b)$ into itself such that
\begin{align*}
\langle L_{\omega}v,\psi\rangle=&B_{\omega}(v,\psi)=\int_{\Omega_b}\left(E(v, \overline{\psi})-(\omega^2-\mu\alpha^{2})v \cdot \overline{\psi}-2i\alpha\mu\frac{\partial v}{\partial x_{1}}\cdot\overline{\psi}\right) \mathrm{d}x-\int_{\Gamma_b} \overline{\psi} \cdot \mathscr{T}_{\omega}v \mathrm{~d}s \\		&-i\alpha(\lambda+\mu)\int_{\Omega_b}\left[\left(\frac{\partial v_{2}}{\partial x_{2}}+2\frac{\partial v_{1}}{\partial x_{1}}\right)\overline{\psi}_{1}+ \frac{\partial v_{1}}{\partial x_{2}}\overline{\psi}_{2}\right] \mathrm{d}x +(\lambda+\mu)\int_{\Omega_b}\alpha^{2}v_{1}\overline{\psi}_{1}\mathrm{d}x  \\
	=&\int_{\Omega_b}\left(E(v,\overline{\psi})-\omega^{2}\cos^{2}\theta\, v \cdot \overline{\psi}-2i\sqrt{\mu}\sin\theta\,\omega\frac{\partial v}{\partial x_{1}}\cdot\overline{\psi}\right) \mathrm{d}x  \\
	&-\frac{i(\lambda+\mu)\sin\theta}{\sqrt{\mu}}\omega \int_{\Omega_b}\left[\left(\frac{\partial v_{2}}{\partial x_{2}}+2\frac{\partial v_{1}}{\partial x_{1}}\right)\overline{\psi}_{1}+ \frac{\partial v_{1}}{\partial x_{2}}\overline{\psi}_{2}\right] \mathrm{d}x  \nonumber\\
	&+\frac{(\lambda+\mu)\sin^{2}\theta}{\mu}\omega^{2} \int_{\Omega_b}v_{1}\overline{\psi}_{1}\mathrm{d}x +2\pi\sum_{n\in\mathbb{Z}} \overline{\psi}_{n} \cdot \left(W_n(\omega) v_n\right),
\end{align*}
where $v_n$ and $w_n$ denote the Fourier coefficients of $v$ and $w$ on $\Gamma_b$, respectively.
The analogue of Lemma \ref{Lem-p} in the shear plane wave case is stated below. 		
\begin{lemma} \label{Lem-s}
Let $\alpha=k_s\sin\theta$, suppose that $|\alpha_{n}|\neq k_{p}$ and $|\alpha_{n}|\neq k_{s}$ for any $n\in\mathbb{Z}$.
			
(i)The null space $\mathcal{N}:=\mathcal{N}(L_{\omega})=\mathcal{N}(L_{\omega}^{*})$ is finite dimensional and consists of exponentially decaying modes.
			
(ii) The Riesz number of $L_{\omega}$ is one, that is, $\mathcal{N}(L_{\omega}) =\mathcal{N}(L_{\omega}^{2})$. Moreover, there holds the orthogonal decomposition $V_{{\rm per}}=\mathcal{N}(L_{\omega})\oplus\mathcal{R}(L_{\omega})$.
			
(iii) If ${\rm Im}\,\omega>0$ and $\sin^{2}\theta<\frac{\mu}{\lambda+2\mu}$, there is a unique solution to \eqref{c0} for any $b>\Lambda^+$.
\end{lemma}
\begin{proof}
The assertions (i) and (ii) can be proved following the same arguments used in the proof of Lemma \ref{Lem-p}. Below we only verify the third assertion.
			
(iii) Denoting $\omega:=\omega_{0}+i\varepsilon$ with $\varepsilon>0$, then
\begin{align*}
	\beta_{n}^{2}=&k_{p}^{2}-\alpha_{n}^{2}=\frac{\omega^{2}}{\lambda+2\mu} -\left(n+\frac{\omega}{\sqrt{\mu}}\sin\theta\right)^{2} \\
	=&\frac{\omega_{0}^{2}-\varepsilon^{2}}{\lambda+2\mu} -\left(n+\frac{\omega_{0}\sin\theta}{\sqrt{\mu}}\right)^{2} +\frac{\varepsilon^{2}\sin^{2}\theta}{\mu} +\frac{2i\varepsilon\omega_{0}}{\lambda+2\mu} -\frac{2i\varepsilon\sin\theta}{\sqrt{\mu}} \left(n+\frac{\omega_{0}\sin\theta}{\sqrt{\mu}}\right),
\end{align*}
and
\begin{align*}
	\gamma_{n}^{2}=&k_{s}^{2}-\alpha_{n}^{2}=\frac{\omega^{2}}{\mu} -\left(n+\frac{\omega}{\sqrt{\mu}}\sin\theta\right)^{2} \\
	=&\frac{\omega_{0}^{2}-\varepsilon^{2}}{\mu} -\left(n+\frac{\omega_{0}\sin\theta}{\sqrt{\mu}}\right)^{2} +\frac{\varepsilon^{2}\sin^{2}\theta}{\mu} +\frac{2i\varepsilon\omega_{0}}{\mu} -\frac{2i\varepsilon\sin\theta}{\sqrt{\mu}} \left(n+\frac{\omega_{0}\sin\theta}{\sqrt{\mu}}\right),
\end{align*}
that is,
\ben
	{\rm Re}\beta_{n}^{2}&=&\frac{\omega_{0}^{2}-\varepsilon^{2}}{\lambda+2\mu} -\left(n+\frac{\omega_{0}\sin\theta}{\sqrt{\mu}}\right)^{2} +\frac{\varepsilon^{2}\sin^{2}\theta}{\mu}, \quad
	{\rm Im}\beta_{n}^{2}=\frac{2\varepsilon\omega_{0}}{\lambda+2\mu} -\frac{2\varepsilon\sin\theta}{\sqrt{\mu}} \left(n+\frac{\omega_{0}\sin\theta}{\sqrt{\mu}}\right),\\
	{\rm Re}\gamma_{n}^{2}&=&\frac{\omega_{0}^{2}-\varepsilon^{2}}{\mu} -\left(n+\frac{\omega_{0}\sin\theta}{\sqrt{\mu}}\right)^{2} +\frac{\varepsilon^{2}\sin^{2}\theta}{\mu}, \quad
	{\rm Im}\gamma_{n}^{2}=\frac{2\varepsilon\omega_{0}}{\mu} -\frac{2\varepsilon\sin\theta}{\sqrt{\mu}} \left(n+\frac{\omega_{0}\sin\theta}{\sqrt{\mu}}\right).
\enn
			
Assume that ${\rm Im}\beta_{n}^{2}\leq0$, then we have
\begin{equation*}
	0<\frac{\varepsilon\omega_{0}}{\lambda+2\mu} \leq \frac{\varepsilon\sin\theta}{\sqrt{\mu}} \left(n+\frac{\omega_{0}\sin\theta}{\sqrt{\mu}}\right),
\end{equation*}
which implies that
\begin{align*}
{\rm Re}\beta_{n}^{2}\leq& \frac{\omega_{0}^{2}-\varepsilon^{2}}{\lambda+2\mu} +\frac{\varepsilon^{2}\sin^{2}\theta}{\mu} -\frac{\mu\omega_{0}^{2}}{(\lambda+2\mu)^{2}\sin^{2}\theta} \\
	=& \frac{\omega_{0}^{2}}{\lambda+2\mu} \left(1-\frac{\mu}{(\lambda+2\mu)\sin^{2}\theta}\right) + \varepsilon^{2}\left(\frac{\sin^{2}\theta}{\mu}-\frac{1}{\lambda+2\mu} \right)<0,
\end{align*}
where we have used the assumption \eqref{assu}.
This proves either ${\rm Re}\beta_{n}^{2}<0$ or ${\rm Im}\beta_{n}^{2}>0$. Hence, one must have ${\rm Im}\beta_{n}>0$.
			
Assume that ${\rm Im}\gamma_{n}^{2}\leq0$, then we have
\begin{equation*}
0<\frac{\varepsilon\omega_{0}}{\mu} \leq\frac{\varepsilon\sin\theta}{\sqrt{\mu}} \left(n+\frac{\omega_{0}\sin\theta}{\sqrt{\mu}}\right),
\end{equation*}
which implies that
\begin{align*}
{\rm Re}\gamma_{n}^{2}\leq\frac{\omega_{0}^{2}-\varepsilon^{2}}{\mu} +\frac{\varepsilon^{2}\sin^{2}\theta}{\mu} -\frac{\omega_{0}^{2}}{\mu\sin^{2}\theta} =\frac{\omega_{0}^{2}}{\mu}\left(1-\frac{1}{\sin^{2}\theta}\right) -\frac{\varepsilon^{2}}{\mu}\cos^{2}\theta <0.
\end{align*}
This proves that ${\rm Im}\, \gamma_{n}>0$ for all $n\in\mathbb{Z}$ for the same reason.
			
Let $v\in V_{{\rm per}}(\Omega_\infty)$ be a solution of the homogeneous problem \eqref{a0}, \eqref{a1}, \eqref{a4} with ${\rm Im}\,\omega>0$. Hence, by the Rayleigh expansion \eqref{a4}, $v$ must exponetially decay in the $x_{2}$-direction, allowing a variational formulation over $V_{{\rm per}}(\Omega_\infty)$:
\begin{align*}
	&\int_{\Omega_\infty}\left(E(v,\overline{\psi})-\omega^{2}\cos^{2}\theta\, v \cdot \overline{\psi}-2i\sqrt{\mu}\sin\theta\,\omega\frac{\partial v}{\partial x_{1}}\cdot\overline{\psi}\right) \mathrm{d}x+\frac{(\lambda+\mu)\sin^{2}\theta}{\mu}\omega^{2} \int_{\Omega_\infty}v_{1}\overline{\psi}_{1}\mathrm{d}x  \\
	&-\frac{i(\lambda+\mu)\sin\theta}{\sqrt{\mu}}\omega \int_{\Omega_\infty}\left[\left(\frac{\partial v_{2}}{\partial x_{2}}+2\frac{\partial v_{1}}{\partial x_{1}}\right)\overline{\psi}_{1}+ \frac{\partial v_{1}}{\partial x_{2}}\overline{\psi}_{2}\right] \mathrm{d}x=0.
\end{align*}
This can be written as
\begin{equation*}
\mathbf{\widetilde{A}}v-\omega \mathbf{\widetilde{B}}v-\omega^{2}\mathbf{\widetilde{C}}v=0
\end{equation*}
where $\mathbf{\widetilde{A}}$, $\mathbf{\widetilde{B}}$ and $\mathbf{\widetilde{C}}$ are defined by
\be \label{A2}
\langle \mathbf{\widetilde{A}}v,\psi\rangle &=&\int_{\Omega_\infty}E(v,\overline{\psi})\,\mathrm{d}x,
	\\  \label{B2}
	\langle \mathbf{\widetilde{B}}v,\psi\rangle &=&2i\sqrt{\mu}\sin\theta\int_{\Omega_\infty}\frac{\partial v}{\partial x_{1}}\cdot\overline{\psi}\,\mathrm{d}x			+\frac{i(\lambda+\mu)\sin\theta}{\sqrt{\mu}}\int_{\Omega_\infty}\left[\left(\frac{\partial v_{2}}{\partial x_{2}}+2\frac{\partial v_{1}}{\partial x_{1}}\right)\overline{\psi}_{1}+ \frac{\partial v_{1}}{\partial x_{2}}\overline{\psi}_{2}\right] \mathrm{d}x,
	\\ \label{C2}
\langle \mathbf{\widetilde{C}}v,\psi\rangle &=&\cos^{2}\theta\int_{\Omega_\infty}v\cdot\overline{\psi} \,\mathrm{d}x-\frac{(\lambda+\mu)\sin^{2}\theta}{\mu} \int_{\Omega_\infty}v_{1}\overline{\psi}_{1}\,\mathrm{d}x.
\en
It is easy to prove that $\mathbf{\widetilde{A}}$, $\mathbf{\widetilde{B}}$ and $\mathbf{\widetilde{C}}$ are self-adjoint bounded operators from $V_{{\rm per}}(\Omega_\infty)$ into itself. Note that $\mathbf{\widetilde{A}}$ is positive, that is, $\langle \mathbf{\widetilde{A}}v, v\rangle>0$ for all $v \neq 0$. Moreover,
\begin{align*}
\langle \mathbf{\widetilde{C}}v,v\rangle=&\cos^{2}\theta\int_{\Omega_\infty}|v|^{2} \mathrm{d}x -\frac{(\lambda+\mu)\sin^{2}\theta}{\mu} \int_{\Omega_\infty}|v_{1}|^{2}\mathrm{d}x \\				=&\frac{\mu-(\lambda+2\mu)\sin^{2}\theta}{\mu}\int_{\Omega_\infty}|v_{1}|^{2}\mathrm{d}x +\cos^{2}\theta\int_{\Omega_\infty}|v_{2}|^{2}\mathrm{d}x>0, \quad \forall\, v\neq0,
\end{align*}
where we have again used the assumption \eqref{assu}. Hence, $\mathbf{\widetilde{C}}$ is also positive. By arguing analogously to the proof of Lemma \ref{Lem-p} (iii), we obtain $v\equiv 0$.
\end{proof}
		
In the following, we set $X=V_{{\rm per}}(\Omega_\infty)$ and denote by $P:X\rightarrow \mathcal{N}$ the projection operator. Write $f(\epsilon)=f_{\omega+i\epsilon}$ and $L(\epsilon)=L_{\omega+i\epsilon}$. For $\psi\in X$, it follows from the definitions of $f_{\omega}$ and $L_{\omega}$ that
\begin{equation} \label{f2}
\langle f(\epsilon),\psi\rangle=-2i\tau_{4}(\omega+i\epsilon) e^{-ib(\omega+i\epsilon)\cos\theta/\sqrt{\mu}} \int_{0}^{2\pi}  \left(\begin{array}{c}
	\tau_{3} \\  \sin\theta
\end{array}\right)\cdot\overline{\psi(x_{1},b)}\,\mathrm{d}x_{1},
\end{equation}
\begin{align}   \label{L2}
\langle L(\epsilon)v,\psi\rangle=&\int_{\Omega_b}\left(E(v,\overline{\psi}) -(\omega+i\epsilon)^{2}\cos^{2}\theta\, v \cdot \overline{\psi}-2i\sqrt{\mu}\sin\theta\,(\omega+i\epsilon)\frac{\partial v}{\partial x_{1}}\cdot\overline{\psi}\right) \mathrm{d}x  \nonumber\\
	&-\frac{i(\lambda+\mu)\sin\theta}{\sqrt{\mu}}(\omega+i\epsilon) \int_{\Omega_b}\left[\left(\frac{\partial v_{2}}{\partial x_{2}}+2\frac{\partial v_{1}}{\partial x_{1}}\right)\overline{\psi}_{1}+ \frac{\partial v_{1}}{\partial x_{2}}\overline{\psi}_{2}\right] \mathrm{d}x  \nonumber\\
	&+\frac{(\lambda+\mu)\sin^{2}\theta}{\mu}(\omega+i\epsilon)^{2} \int_{\Omega_b}v_{1}\overline{\psi}_{1}\mathrm{d}x +2\pi\sum_{n\in\mathbb{Z}} \overline{\psi}_{n} \cdot \left(W_n(\omega+i\epsilon) v_n\right).
\end{align}
In order to apply Lemma \ref{LAP} to the shear plane wave case, we still need to prove the following results.
\begin{lemma}\label{lem5}
(i) $f(\epsilon)\in\mathcal{R}(L(\epsilon))$ for all $\epsilon>0$ and $f(0)$, $f'(0)\in \mathcal{R}(L(0))$.
			
(ii) $PL'(0)$ is one-to-one on $\mathcal{N}(L(0))$.
\end{lemma}
\begin{proof}
By the definitions \eqref{f2} and \eqref{L2}, we have
\begin{equation*}
\langle f'(\epsilon),\psi\rangle=-2i\tau_{4}\left(i+\frac{\omega+i\epsilon}{\sqrt{\mu}} b\cos\theta\right) e^{-ib(\omega+i\epsilon)\cos\theta/\sqrt{\mu}} \int_{0}^{2\pi}  \left(\begin{array}{c}
		\tau_{3} \\  \sin\theta		\end{array}\right)\cdot\overline{\psi(x_{1},b)}\,\mathrm{d}x_{1},
\end{equation*}
where $\tau_3$ and $\tau_4$ are defined in \eqref{tau34}, and
\begin{align*}
\langle L'(\epsilon)v,\psi\rangle=&\int_{\Omega_b}\left( -2i(\omega+i\epsilon)\cos^{2}\theta\, v \cdot \overline{\psi}+2\sqrt{\mu}\sin\theta\frac{\partial v}{\partial x_{1}}\cdot\overline{\psi}\right) \mathrm{d}x  \\
		&+\frac{(\lambda+\mu)\sin\theta}{\sqrt{\mu}} \int_{\Omega_b}\left[\left(\frac{\partial v_{2}}{\partial x_{2}}+2\frac{\partial v_{1}}{\partial x_{1}}\right)\overline{\psi}_{1}+ \frac{\partial v_{1}}{\partial x_{2}}\overline{\psi}_{2}\right] \mathrm{d}x  \\
		&+\frac{(\lambda+\mu)\sin^{2}\theta}{\mu}2i(\omega+i\epsilon) \int_{\Omega_b}v_{1}\overline{\psi}_{1}\mathrm{d}x +2\pi i\sum_{n\in\mathbb{Z}} \overline{\psi}_{n} \cdot \left(W_n'(\omega+i\epsilon) v_n\right).
\end{align*}
				
(i) By Lemma \ref{Lem-p} (iii) and the Fredholmness of the operator $L(\epsilon)$, the operator $L(\epsilon)$ must be invertible and thus $f(\epsilon)\in \mathcal{R}(L(\epsilon))$ for all $\epsilon>0$. On the other hand, we have $f(0)\in \mathcal{R}(L(0))$, because by Lemma \ref{Lem-p} (i) and (ii), $f(0)$ is orthogonal to $\mathcal{N}(L_{\omega})$ and $V_{{\rm per}}(\Omega_b)$ admits the orthogonal decomposition $V_{{\rm per}}(\Omega_b)=\mathcal{N}(L(0)) \oplus \mathcal{R}(L(0))$.
				
Since $(I-P)\phi$ is orthogonal to $\mathcal{N}(L(0))$ for any $\phi\in V_{{\rm per}}(\Omega_b)$, it follows that
\begin{equation*}
\langle Pf(0),\psi\rangle=\langle f(0),\psi\rangle=-2i\tau_{4}\omega e^{-ib\omega\cos\theta/\sqrt{\mu}} \int_{0}^{2\pi}  \left(\begin{array}{c}
\tau_{3} \\  \sin\theta
\end{array}\right)\cdot\overline{\psi(x_{1},b)}\,\mathrm{d}x_{1}=0,
\end{equation*}
and
\begin{align*}
&\langle Pf'(0),\psi\rangle =\langle f'(0),\psi\rangle  \\ =&-2i\tau_{4}\left(i+\frac{\omega}{\sqrt{\mu}} b\cos\theta\right) e^{-ib\omega\cos\theta/\sqrt{\mu}} \int_{0}^{2\pi}  \left(\begin{array}{c}
	\tau_{3} \\  \sin\theta
\end{array}\right)\cdot\overline{\psi(x_{1},b)}\,\mathrm{d}x_{1}=0,
\end{align*}
for all $\psi\in \mathcal{N}(L(0))$. This implies that $f(0), f'(0)\in\mathcal{R}(L(0))$ and $Pf(0)=Pf'(0)=0$.
				
(ii) We continue to compute
\begin{align}\nonumber
&\langle PL'(0)v,\psi\rangle=\langle L'(0)v,\psi\rangle  \\ \nonumber
=&\int_{\Omega_\infty}\left( -2i\omega\cos^{2}\theta\, v \cdot \overline{\psi}+2\sqrt{\mu}\sin\theta\frac{\partial v}{\partial x_{1}}\cdot\overline{\psi}\right) \mathrm{d}x+\frac{(\lambda+\mu)\sin^{2}\theta}{\mu}2i\omega \int_{\Omega_\infty}v_{1}\overline{\psi}_{1}\mathrm{d}x  \\ \label{or-shear}
	&+\frac{(\lambda+\mu)\sin\theta}{\sqrt{\mu}} \int_{\Omega_\infty}\left[\left(\frac{\partial v_{2}}{\partial x_{2}}+2\frac{\partial v_{1}}{\partial x_{1}}\right)\overline{\psi}_{1}+ \frac{\partial v_{1}}{\partial x_{2}}\overline{\psi}_{2}\right] \mathrm{d}x ,
\end{align}
for all $v,\psi\in \mathcal{N}(L(0))$.
				
Now assume that $\langle PL'(0)v,\cdot\rangle=0$ for some $v\in \mathcal{N}(L(0))$. Then
\begin{align*}
\langle PL'(0)v,v\rangle=&\int_{\Omega_\infty}\left(-2i\omega\cos^{2}\theta\, |v|^{2} +2\sqrt{\mu}\sin\theta\frac{\partial v}{\partial x_{1}}\cdot\overline{v}\right) \mathrm{d}x+\frac{(\lambda+\mu)\sin^{2}\theta}{\mu}2i\omega \int_{\Omega_\infty}|v_{1}|^{2}\mathrm{d}x  \\
	&+\frac{(\lambda+\mu)\sin\theta}{\sqrt{\mu}} \int_{\Omega_\infty}\left[\left(\frac{\partial v_{2}}{\partial x_{2}}+2\frac{\partial v_{1}}{\partial x_{1}}\right)\overline{v}_{1}+ \frac{\partial v_{1}}{\partial x_{2}}\overline{v}_{2}\right] \mathrm{d}x=0 .
\end{align*}
Recalling the definitions of $\mathbf{\widetilde{A}}$, $\mathbf{\widetilde{B}}$ and $\mathbf{\widetilde{C}}$ in \eqref{A2}-\eqref{C2}, the above relation can be rewritten as
\begin{equation*}
\langle \mathbf{\widetilde{B}}v,v\rangle+2\omega\langle \mathbf{\widetilde{C}}v,v\rangle=0.
\end{equation*}
Since $v\in\mathcal{N}(L(0))$ is a solution of the homogeneous problem, we have $\mathbf{\widetilde{A}}v-\omega \mathbf{\widetilde{B}}v-\omega^{2}\mathbf{\widetilde{C}}v=0$. Hence,
\begin{equation*}
\langle \mathbf{\widetilde{A}}v,v\rangle+\omega^{2}\langle \mathbf{\widetilde{C}}v,v\rangle=0.
\end{equation*}
Since $\mathbf{\widetilde{A}}$ and $\mathbf{\widetilde{C}}$ are positive operators, we obtain that $v=0$, which proves that $PL'(0)$ is one-to-one on $\mathcal{N}(L(0))$.
\end{proof}
Like the case of an incident pressure plane wave, the integral identity \eqref{or-shear} will yield an additional constraint condition to ensure uniqueness. Next we transform this identity to quasi-periodic settings by  substituting $v:=e^{-ik_{s}\sin\theta x_{1}}u$. It then follows that
\begin{align*}
&\int_{\Omega_\infty}(-2i\omega\cos^{2}\theta\, u\cdot \overline{\varphi} +2\sqrt{\mu}\sin\theta\frac{\partial u}{\partial x_{1}}\cdot\overline{\varphi}-2i\omega\sin^{2}\theta u\cdot \overline{\varphi} ) \mathrm{d}x +\frac{(\lambda+\mu)\sin^{2}\theta}{\mu}2i\omega\int_{\Omega_\infty}u_{1}\overline{\varphi}_{1} \mathrm{d}x \\ &+\frac{(\lambda+\mu)\sin\theta}{\sqrt{\mu}} \int_{\Omega_\infty}\left[\left(\frac{\partial u_{2}}{\partial x_{2}}+2\frac{\partial u_{1}}{\partial x_{1}}\right)\overline{\varphi}_{1}+ \frac{\partial u_{1}}{\partial x_{2}}\overline{\varphi}_{2}-2i\omega\frac{\sin\theta}{\sqrt{\mu}}u_{1}\overline{\varphi}_{1} \right] \mathrm{d}x=0,
\end{align*}
that is,
\begin{equation*}
\frac{\sin\theta}{\sqrt{\mu}}\int_{\Omega_\infty} \left\{2\mu\frac{\partial u}{\partial x_{1}}\cdot\overline{\varphi}+(\lambda+\mu) \left[\left(\frac{\partial u_{2}}{\partial x_{2}}+2\frac{\partial u_{1}}{\partial x_{1}}\right)\overline{\varphi}_{1}+ \frac{\partial u_{1}}{\partial x_{2}}\overline{\varphi}_{2}\right] \right\} \mathrm{d}x=2i\omega\int_{\Omega_\infty}u\cdot\overline{\varphi}\,\mathrm{d}x.
\end{equation*}
Multiplying $k_s$ on both sides and using $\alpha=k_s\sin\theta$, we then obtain the same integral identity as \eqref{integral-p}. Finally, combining Lemmas \ref{Lem-s}, \ref{lem5} and \ref{LAP} gives the limiting absorption principle for the incident shear plane wave as follows.
\begin{theorem}\label{THS}
Let $\omega>0$ be fixed and $\theta\in(-\pi/2,\pi/2)$ be an arbitrary incident direction. Set $\alpha=k_{s}\sin\theta$ and suppose that $|\alpha+n|\neq k_{p}$, $|\alpha+n|\neq k_{s}$ for any $n\in\mathbb{Z}$. There exists a unique solution $u\in [H_{{\rm loc}, \alpha,0}^{1}(D) ]^{2}$ such that $u^{\rm sc}:=u-u^{\rm in}$ satisfies the $\alpha$-quasiperiodic Rayleigh expansion if $\alpha$ is not a propagative wave number. At a propagative wavenumber, there still exists a unique solution if we additionally enforce the following constraint condition
\begin{equation}\label{integral-s}
\alpha\,\int_{\Omega_\infty} \left\{2\mu\frac{\partial u}{\partial x_{1}}\cdot\overline{\varphi}+(\lambda+\mu) \left[\left(\frac{\partial u_{2}}{\partial x_{2}}+2\frac{\partial u_{1}}{\partial x_{1}}\right)\overline{\varphi}_{1}+ \frac{\partial u_{1}}{\partial x_{2}}\overline{\varphi}_{2}\right] \right\} \mathrm{d}x=2i\omega^2\int_{\Omega_\infty}u\cdot\overline{\varphi}\,\mathrm{d}x,
\end{equation}
for all $\varphi\in[H_{{\rm loc}, \alpha,0}^{1}(D) ]^{2}$ such that $e^{-ik_{s}\sin\theta\, x_{1}}\varphi \in \mathcal{N}$.
\end{theorem}

\vspace{0.8em}
Finally, we end up this paper with several remarks.

\begin{remark}\label{rem2}
Consider a mixed incident plane wave of the form
$$u^{{\rm in}}=c_p\,u^{{\rm in}}_p + c_s\,u^{{\rm in}}_s,\quad  |c_p|+|c_s|\neq 0. $$
Suppose that $\sin^{2}\theta<\frac{\mu}{\lambda+2\mu}$ and
\ben
|k_p\sin\theta+n|\neq k_{p},\quad |k_p\sin\theta+n|\neq k_{s}, \\
|k_s\sin\theta+n|\neq k_{p},\quad |k_s\sin\theta+n|\neq k_{s},
\enn
for all $n\in \Z$. Then there exists a unique $k_p\sin\theta$-quasiperiodic solution $u_p$ given by Theorem \ref{THP} and corresponding to $u^{{\rm in}}_p$ and  a unique $k_s\sin\theta$-quasiperiodic solution $u_s$ given by Theorem \ref{THS} and corresponding to $u^{{\rm in}}_s$. Hence, $u:=c_p u_p + c_s u_s$ is the unique solution under the additional constraint conditions \eqref{integral-p} and \eqref{integral-s} enforcing on $u_p$ and $u_s$.  Notice that $u_p$ and $u_s$ have different phase shifts.
\end{remark}

\begin{remark}\label{rem3}
All results of this paper carry over to the case of the cavity (Neumann) boundary condition and other transmission problems in linear elasticity. One needs necessarily to change the variational space and in the transmission case that the definition of $\Omega_\infty$. In addition, our arguments extend to  bi-periodic structures in three dimensions as well. We omit the detailed analysis for brevity.
\end{remark}

\begin{remark}
The orthogonality conditions \eqref{integral-s} and \eqref{integral-p} furnish a criterion for the numerical determination of the coefficients $c_j$
in the general solution form \eqref{general} of our diffraction problem at a propagating wavenumber. These, together with the Rayleigh expansion condition \eqref{a4}, constitute a sharp radiation condition that guarantees the well-posedness of the diffraction model.	
	\end{remark}
					
\vspace{0.8em}

\textbf{Acknowledgement}
\vspace{0.5em}

The work of G.H. Hu is supported by the National Natural Science Foundation of China (No. 12425112), the Fundamental Research Funds for Central Universities in China (No. 050-63263073) and the Natural Science Foundation of Tianjin (No. 25JCZDJC00970). The work of J.L. Xiang is supported by the Natural Science Foundation of China (No. 12301542) and the Open Research Fund of Hubei Key Laboratory of Mathematical Sciences (Central China Normal University, MPL2025ORG017).

\end{document}